%% file: main_iz.tex
\documentclass[12pt,reqno]{amsart}

\numberwithin{equation}{section}
\usepackage[tt=false]{libertine}
\usepackage{mathtools}
\usepackage{thm-restate}
\usepackage{xspace}

\usepackage{amssymb,mathrsfs}
\usepackage[varbb]{newpxmath}

\let\savedbigtimes\bigtimes
\let\bigtimes\relax
\usepackage{mathabx} 
\let\bigtimes\savedbigtimes

\usepackage[margin=1in]{geometry}
\usepackage{graphicx}
\usepackage{enumerate}
\usepackage{enumitem}
\usepackage{bbm}
\usepackage{hyperref,color}
\usepackage[capitalize,nameinlink]{cleveref}
\usepackage[dvipsnames]{xcolor}
\hypersetup{
	colorlinks=true,
	pdfpagemode=UseNone,
    citecolor=OliveGreen,
    linkcolor=NavyBlue,
    urlcolor=black,
	pdfstartview=FitW
}
\usepackage{appendix}
\crefname{appsec}{Appendix}{Appendices}
\usepackage{tikz}
\usepackage{bm}
\usepackage{mathtools}

\newtheorem{theorem}{Theorem}[section]

\newtheorem{lemma}[theorem]{Lemma}
\newtheorem{corollary}[theorem]{Corollary}

\theoremstyle{definition}
\newtheorem{definition}[theorem]{Definition}

\newtheorem*{assumption*}{Assumption}
\newtheorem{remark}[theorem]{Remark}

\crefname{lemma}{Lemma}{Lemmas}
\crefname{theorem}{Theorem}{Theorems}
\crefname{definition}{Definition}{Definitions}
\crefname{fact}{Fact}{Facts}
\crefname{claim}{Claim}{Claims}
\crefname{proposition}{Proposition}{Propositions}

\newcommand{\E}{\mathbb{E}}

\DeclareMathOperator*{\argmax}{arg\,max}

\newcommand{\norm}[1]{\left\lVert #1 \right\rVert}

\newcommand{\eps}{\varepsilon}
\renewcommand{\epsilon}{\varepsilon}

\newcommand{\R}{\mathbb{R}}

\renewcommand{\P}{\mathbb{P}}

\newcommand{\beq}{\begin{equation}}
\newcommand{\eeq}{\end{equation}}

\newcommand{\Ind}[1]{\mathbf{1}\{#1\}}

\newcommand{\bG}{\bm{G}}

\usepackage{appendix}
\crefname{appsec}{Appendix}{Appendices}

\renewcommand{\le}{\leqslant}
\renewcommand{\leq}{\leqslant}
\renewcommand{\ge}{\geqslant}
\renewcommand{\geq}{\geqslant}

\renewcommand{\emph}{\textit}

\usepackage{subcaption}

\newcommand{\erdos}{Erd\H{o}s\xspace}
\newcommand{\renyi}{R\'enyi\xspace}
\newcommand{\MMSE}{\operatorname{MMSE}}

\newcommand{\defeq}{\coloneqq}

\newcommand{\parhead}[1]{\medskip \noindent {\bfseries\boldmath\ignorespaces {#1}}\hskip 0.9em}

\newcommand{\Ber}{\operatorname{Ber}}

\newcommand{\wt}[1]{\widetilde{#1}}
\renewcommand{\P}{\mathbb{P}}

\newcommand{\mrm}[1]{\mathrm{#1}}

\newcommand{\lprp}[1]{\left(#1\right)}

\newcommand{\bone}{\mathbf{1}}

\renewcommand{\E}{\mathbb{E}}

\newcommand{\calA}{\mathcal A}

\newcommand{\calE}{\mathcal E}

\newcommand{\calG}{\mathcal G}

\newcommand{\calX}{\mathcal X}

\renewcommand{\le}{\leqslant}
\renewcommand{\leq}{\leqslant}
\renewcommand{\ge}{\geqslant}
\renewcommand{\geq}{\geqslant}
\renewcommand{\phi}{\varphi}

\begin{document}

\title[The Fundamental Limits of Recovering Planted Subgraphs]{The Fundamental Limits of Recovering Planted Subgraphs}

\author[D.Z. Lee, F. Pernice, A. Rajaraman, I. Zadik]{Daniel Z. Lee$^\circ$
  \and Francisco Pernice$^\circ$
  \and Amit Rajaraman$^\circ$
  \and Ilias Zadik$^\dagger$}
  \thanks{\raggedright$^\circ$CSAIL, Massachusetts Institute of Technology;
$^\dagger$Department of Statistics and Data Science, Yale University.
Email: \texttt{\{lee\_d,fpernice,amit\_r\}@mit.edu, ilias.zadik@yale.edu}}

\begin{abstract}
   Given an arbitrary subgraph $H=H_n$ and $p=p_n \in (0,1)$, the planted subgraph model is defined as follows. A statistician observes the union of the ``signal,'' which is a random ``planted'' copy $H^*$ of $H$, together with random noise in the form of an instance of an \erdos--\renyi graph $G(n,p)$. Their goal is to then recover the planted $H^*$ from the observed graph. Our focus in this work is to understand the minimum mean squared error (MMSE), defined in terms of recovering the edges of $H^*$, as a function of $p$ and $H$, for sufficiently large $n$.

    A recent paper \cite{MNSSZ23} characterizes the graphs for which the limiting (as $n$ grows) MMSE curve undergoes a sharp phase transition from $0$ to $1$ as $p$ increases, a behavior known as the all-or-nothing phenomenon, up to a mild density assumption on $H$. However, their techniques fail to describe the MMSE curves for graphs that do not display such a sharp phase transition. In this paper, we provide a formula for the limiting MMSE curve for \emph{any} graph $H=H_n$, up to the same mild density assumption. This curve is expressed in terms of a variational formula over pairs of subgraphs of $H$, and is inspired by the celebrated subgraph expectation thresholds from the probabilistic combinatorics literature \cite{KK07}. Furthermore, we give a polynomial-time description of the optimizers of this variational problem. This allows one to efficiently approximately compute the MMSE curve for any dense graph $H$ when $n$ is large enough. The proof relies on a novel graph decomposition of $H$ as well as a new minimax theorem which may be of independent interest.

    Our results generalize to the setting of minimax rates of recovering arbitrary monotone boolean properties planted in random noise, where the statistician observes the union of a planted minimal element $A \subseteq [N]$ of a monotone property and a random $\Ber(p)^{\otimes N}$ vector. In this setting, we provide a variational formula inspired by the so-called ``fractional'' expectation threshold \cite{Tal10}, again describing the MMSE curve (in this case up to a multiplicative constant) for large enough $n$.
\end{abstract}

\maketitle

\date{\today}

\thispagestyle{empty}
\setcounter{page}{0}



\newpage


\input{intro3_iz}

\input{non-sparse_iz}

\input{appendix_iz}

\newpage

\bibliographystyle{alpha}
\bibliography{refs}

\newpage

\input{real_appendix_iz}
\end{document}

%% file: intro3_iz.tex

\tableofcontents
\newpage \section{Introduction}

Over the past two decades, several influential papers have addressed the limiting (normalized) minimum mean squared error (MMSE) in high-dimensional statistical models, including \cite{dia2016mutual,lelarge2017fundamental,montanari2006analysis,reeves2019replica} among others. In this paper, we focus on understanding the limiting MMSE in the fundamental problem of recovering planted discrete structures from noisy random environments.

\parhead{A motivating example: planted clique.} A very well-studied discrete high-dimensional model is the planted clique model, initially defined independently by Jerrum \cite{jerrum1992large} and Ku\v{c}era \cite{kuvcera1995expected}. To define it, we need two parameters, $k=k_n \in \mathbb{N}$ and $p=p_n \in (0,1)$. We assume that the statistician observes the union of a ``signal,'' which is a $k$-clique chosen uniformly at random from the complete graph $K_n$, and the ``noise,'' which is an instance of an \erdos--\renyi graph $G(n,p)$. The goal of the statistician is to recover the planted clique from the observed graph. It should be noted that there is a large body of work dedicated to understanding the special case where $p=1/2$ and $k=O(\sqrt{n})$. However, it is beneficial for us here to consider the more general case where $k=k_n$ and $p=p_n$ can scale arbitrarily with $n$. This generality has recently proven useful in circuit lower bounds \cite{gamarnik2023sharp} and low-degree lower bounds \cite{xifan24}, among others.

The planted clique model has been studied extensively due to its conjectured computational hardness in certain regimes, and has interesting implications for other high-dimensional statistical models (see, e.g., \cite{brennan2018reducibility} and references therein). From a purely statistical perspective, \cite{MNSSZ23} recently computed the limiting MMSE of the model as $n$ grows to infinity, proving that it converges to a step function, a behavior known as the All-or-Nothing (AoN) phenomenon (initially discovered in the context of Gaussian sparse regression, first conjectured to hold in \cite{gamarnik2022sparse} and proven in \cite{reeves2020all}).

\parhead{The planted subgraph model.} Motivated by the extensive study of models similar to the planted clique model, such as the planted matching model \cite{moharrami2021planted,ding2023planted} and the planted dense subgraph model \cite{dhawan2025detection}, the authors of \cite{MNSSZ23} initiated the study of the limiting MMSE for general planted subgraphs (also independently introduced in \cite{huleihel2022inferring} and \cite{abram2023cryptography} by the cryptography community). To define this model, for each $n \in \mathbb{N}$, fix \emph{any} (unlabeled) subgraph $H=H_n$
of the complete graph $K_n$ and some $p=p_n \in (0,1)$.
Next, draw a uniformly random copy $H^*$ of $H$ in the complete graph.
The statistician observes the graph $G$ on $n$ vertices which is the edge union of the ``signal'' $H^*$ with ``noise,'' which is an instance of $G_0 \sim G(n,p),$ i.e., $G=H^* \cup G_0$. It is easy to verify that this more general model subsumes the problems of planted clique and planted perfect matching as special cases.

Following \cite{MNSSZ23}, the main object of interest for us is the (normalized) MMSE of the model, defined as
\begin{equation}
    \label{eq:def-mmse}
    \mathrm{MMSE}_n(p)=\frac{1}{|H|}\min_{\hat{A}} \E \left[ \|\mathbf{1}_{H^*}-\hat{A}(G)\|^2_2\right]=\frac{1}{|H|}\mathbb{E}\left[ \| \mathbf{1}_{H^*}-\mathbb{E}[\mathbf{1}_{H^*}|G]\|^2_2\right],
\end{equation}
where by $\mathbf{1}_{H^*}$ we refer the indicator vector of the edges of $H^*$, $\mathbf{1}_{H*} \in \{0,1\}^{\binom{n}{2}}$ \footnote{Throughout the paper we identify a graph $H$ with its edge set. Hence, by $J \subseteq H$ we refer to the edge-induced subgraph $J$ of $H$. Moreover, note that the MMSE \eqref{eq:def-mmse} in the planted subgraph model is measured in terms of number of edges recovered.}. Observe that $\MMSE_n(p) \in [0,1]$, since the trivial estimator $\hat{A}=0$ attains mean squared error $1$. 


The main result of \cite{MNSSZ23} is that for any ``weakly dense''\footnote{This result concerns graphs $H$ with $|H|$ edges and $v(H)$ vertices that satisfy $|H| =\omega(v(H)\log v(H)).$} $H=H_n,$ the planted subgraph model for $H=H_n$ satisfies the AoN phenomenon (i.e., the MMSE converges to a step function) if and only if the graph $H$ is balanced, in that $\max_{J \subseteq H} \frac{|J|}{v(J)} = \frac{|H|}{v(H)}.$\footnote{Technically, \cite{MNSSZ23} proves this equivalence for an asymptotic definition of balancedness, but we omit the details for the introduction.}
The main goal of this paper is to go beyond the AoN phenomenon, and characterize the MMSE curve for an \emph{arbitrary} weakly dense $H$, for sufficiently large $n$.


\parhead{Kahn--Kalai thresholds and Bayesian approaches.}
An easy application of Bayes' rule shows that the posterior distribution of the planted subgraph model is the uniform distribution among all the copies of $H=H_n$ in the observed graph $G=H^* \cup G_0, G_0 \sim G(n,p)$. As a result, to understand the MMSE one needs to accurately calculate appropriate subgraph counts in $G$; specifically, those of the copies of $H$ in $G$ that have a given overlap with the planted copy $H^*$.

Because the planted graph $H=H_n$ may be arbitrary, na\"{i}ve first and second moment method arguments fail to provide good estimates for these subgraph counts. To circumvent this issue, we derive inspiration from probabilistic combinatorics in the study of the ``null" model $G(n,p)$, specifically the ``Kahn--Kalai'' conjectures on the tightness of expectation thresholds \cite{KK07,Tal10,FKP21,PP24}.

The expectation thresholds relate to the following question about subgraph counts: given an (unlabeled) graph $H$, at what critical value of $p = p_c(H)$ does a copy of $H$ begin to occur in $G(n,p)$ with high probability?\footnote{More concretely, what is the value of $p_c$ such that $\mathbb{P}_{G \sim G(n,p_c)}\left[ \text{a copy of } H \text{ appears in } G \right] = 1/2$?}
This has been studied for specific choices of $H$ as far back as the original works of \erdos and \renyi \cite{ER59,ER60}, and has spurred a rich body of work. A series of conjectures by Kahn and Kalai, later refined by Talagrand, provide simple formulas that are conjectured to yield the critical threshold $p_c$ up to a multiplicative $O(\log |H|)$ factor.

In particular, a conjecture of Kahn and Kalai, which remains unproven, asserts that the critical threshold for any $H$ is essentially given by its \emph{subgraph expectation threshold}, $\max_{S \subseteq H} p_{\mathrm{1M}}(S),$ up to a multiplicative $O(\log |H|)$ factor. Here, $p_{\mathrm{1M}}(S)$ is the first moment threshold for the appearance of the subgraph $S$. A straightforward first moment calculation implies that $p_{\mathrm{1M}}(S)$ is approximately $n^{-v(S)/|S|}$ for weakly dense $S$.
More sophisticated first moment methods result in the \emph{expectation threshold} \cite{KK07} and the \emph{fractional expectation} threshold \cite{Tal10}, which are now known to yield the critical threshold up to a multiplicative $O(\log |H|)$ due to the breakthrough works \cite{FKP21,PP24}.
It should also be noted that these results generalize well beyond subgraph inclusion properties, and identify the critical thresholds for arbitrary monotone properties on a random discrete universe of points under the product Bernoulli measure.

While the original proof of tightness of the fractional expectation threshold in \cite{FKP21} proceeds by careful combinatorial arguments, \cite{mossel2022second} provides an alternate Bayesian proof using an approach similar to that employed in \cite{MNSSZ23} to prove the AoN characterization of the MMSE in the planted subgraph model.
In fact, \cite{mossel2022second} proves the tightness of the fractional expectation threshold for any $H$, by bounding the limiting MMSE of the planted subgraph model corresponding to $H$ from below. This leads to the following question:
\begin{quote}
\centering 
\emph{Can one go the other direction, and use the expectation thresholds to characterize the MMSE for a large family of subgraph models?}
\end{quote}
Indeed, in this work, we prove that for large $n$, the MMSE curve is approximately given by a piecewise-constant function, whose discontinuity points are precisely modifications of the (fractional/subgraph) expectation thresholds of $H$.

\subsection{Main Contributions}
As mentioned, our main result is a characterization of the MMSE for the planted subgraph model for all  weakly dense graphs $H=H_n$ (for large enough $n$). To present the theorem, recall, as mentioned above, that the subgraph expectation threshold of any weakly dense $H$ from \cite{KK07} is (approximately) given by $\max_{J \subseteq H} n^{-v(J)/|J|}$ (see \Cref{lem:1mm_wek} for an exact statement and a proof of this simple fact). Motivated by this, we define, for any $q \in [0,1)$, the $q$-modified subgraph expectation threshold of $H$:
\begin{equation}
    \label{eq:phi-q}
   \varphi_{q} = (\varphi_{q})_n=\min\left\{ \max\left\{ n^{-\frac{|V(J) \setminus V(S)|}{|J \setminus S|}} : J \subseteq H , J \supsetneq S \right\} : S \subseteq H, |S| \le q |H| \right\}.
\end{equation}
Also set $\varphi_1 \defeq 0$.
In words, we define $\varphi_q$ as the minimum possible subgraph expectation threshold of $H \setminus S$ among all choices of $S$ with at most $q |H|$ edges.\footnote{Here, by $H \setminus S$, we refer to the ``graph cut'' constructed after we delete all edges from $S$ from $H$: see \Cref{dfn:graph_cut} for a precise statement.}

Why are the thresholds $(\varphi_q)_{q \in [0,1]}$ potentially relevant for the MMSE curve?
Assuming the validity of the subgraph version of the Kahn--Kalai conjecture, $\varphi_q$ approximately corresponds to the critical density $p$ such that for some $S \subseteq H$ with at most $q |H|$ edges, the graph $H \setminus S$ appears in $G(n,p)$ with high probability. In particular, in the planted model, this is precisely the threshold at which a copy of $H$ appears in the noise that intersects $H^*$ at some subgraph $S$ with at most $q |H|$ edges. This suggests that $\varphi_q$ could be the threshold such that, as soon as $p>\varphi_q$, it is impossible to recover more than a $q$ fraction of the edges of $H^*$ because another copy of $H$ appears in $G$ with $|H \cap H^*| \leq q |H|$. Then, using standard Bayesian techniques, the above can be equivalently phrased as $p=\varphi_q$ being the threshold at which the MMSE hits $1-q$. 

However, this line of argument cannot be applied directly to conclude the desired result. Perhaps the most crucial issue is that the subgraph version of the Kahn--Kalai conjecture remains open for a growing $H=H_n$ (even restricted to weakly dense $H$); moreover, even if proved, it is only expected to approximate $p_c(H)$ to within a multiplicative $O(\log |H|)$ factor of the true critical threshold. Despite such barriers, we are able to prove that the threshold $(\varphi_q)_{q \in [0,1]}$ indeed characterize (with no multiplicative slack) the points that the MMSE crosses the values $(1-q)_{q \in [0,1]}$ in the planted subgraph model, confirming the above intuition. 

More specifically, our main result can be informally stated as follows.

\begin{theorem}[Informal version of \Cref{thm:wek_dense} and \Cref{lem:onion-universality}(c)]\label{thm:inf_main}Let $\eps > 0$, and $H = H_n$ an arbitrary weakly dense graph. Then, for sufficiently large $n$, there exists an integer $1\leq M\leq |H|$ and thresholds $1=q_M>\cdots>q_1>q_0=0$, such that the following holds for the planted subgraph model corresponding to $H$.
    \begin{itemize}
        \item For all $i=0,1,\ldots,M-1$, if $p \in ((1+\eps)\varphi_{q_{i+1}},(1-\eps) \varphi_{q_{i}})$, it holds that
        \[\mathrm{MMSE}_n(p)=1-q_{i+1}+o(1).\]
        For $p \ge (1+\eps) \varphi_{q_0}$, $\MMSE_n(p) = 1-o(1)$.

        \item The ``discontinuity" points $(\varphi_{q_i})_{i=1,2,\ldots,M}$ can be computed in $\mathrm{poly}(n)$ time.
    \end{itemize}  
\end{theorem}

Our proof is in fact able to provide an intuitive characterization of the ``critical'' subgraphs of $H^*$ that are recoverable from $G$ as one gradually decreases the noise level $p$ from one to zero, revealing that they are in fact highly structured.  The first subgraph of $H^*$ that is recoverable is (perhaps naturally) the densest one, i.e., the subgraph $J^{1}$ attaining the maximum $\max_{J \subseteq H} |J|/v(J)$, which happens at the critical threshold $p=\varphi_{0}=n^{- v(J^{1})/|J^{1}|}$. This results in the first ``jump" of the MMSE from $1$ to $1-q_1$ where $q_1=\frac{|J^{1}|}{|H|}$; in particular, if $p<\varphi_{q_1}$, a $q_1$-fraction of the edges of $H$ (corresponding to the subset $J^{(1)}:=J^{1}$) can be recovered. As we continue decreasing $p$, the next critical subgraph is the densest one after deleting the graph $J^{1}$ from $H$. This produces the next jump in the MMSE at critical probability $p=\varphi_{q_1}=\max_{J \subseteq H \setminus J^{1}} n^{-v(J)/|J|}=n^{-v(J^{2})/|J^{2}|}$ from $1-q_1$ to $1-q_2$ for $q_2=\frac{|J^{2}|+|J^{1}|}{|H|}$; in particular, if $p<\varphi_{q_2}$, a $q_2$-fraction of the edges of $H$ (corresponding to $J^{(2)}\defeq J^1 \cup J^2$) can be recovered. Continuing this ``peeling process'' of $H$ produces a sequence of these critical subgraphs $J^{(i)}, i=1,2,\ldots,M$ which we define later as the ``onion decomposition'' of $H$ (see \Cref{fig:onion} for a pictorial depiction of these thresholds, and \Cref{dfn:onion} for more precise details). We direct the reader to \Cref{cor:wek-dense-improved} and the discussion after it for more precise statements.

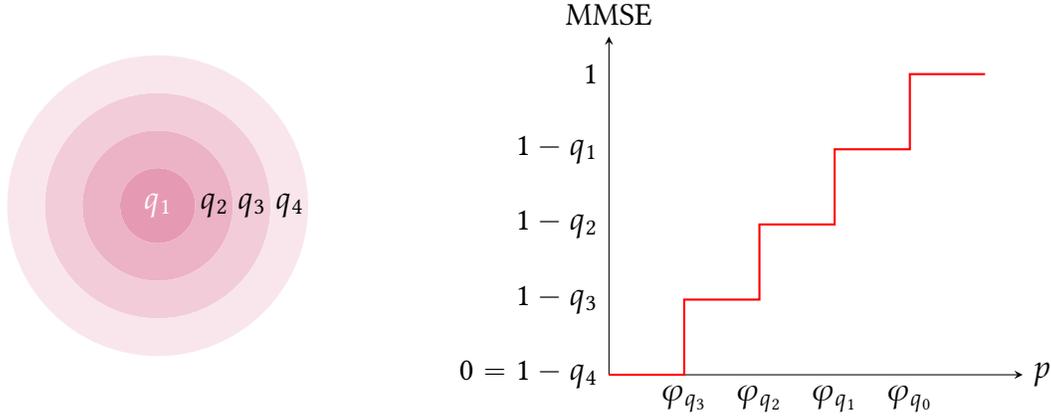
\begin{figure}

\centering

\begin{tikzpicture}[>=stealth]

  \begin{scope}[yshift=2.25cm]
    \fill[purple!40] (0,0) circle (0.5);
    \node[color=white] at (0,0) {$q_1$};

    \path[fill=purple!30,even odd rule] (0,0) circle (1) (0,0) circle (0.5);
    \node at (0.75,0) {$q_2$};

    \path[fill=purple!20,even odd rule] (0,0) circle (1.5) (0,0) circle (1);
    \node at (1.25,0) {$q_3$};

    \path[fill=purple!10,even odd rule] (0,0) circle (2) (0,0) circle (1.5);
    \node at (1.75,0) {$q_4$};
  \end{scope}

  \begin{scope}[xshift=6cm]
    \draw[->] (0,0) -- (5.5,0) node[right] {$p$};
    \draw[->] (0,0) -- (0,4.5) node[above] {MMSE};
    
    \draw[red,thick]
      (0,0)        -- (1,0)
      -- (1,1)     -- (2,1)
      -- (2,2)     -- (3,2)
      -- (3,3)     -- (4,3)
      -- (4,4)     -- (5,4);

    \node[below] at (1,0) {$\varphi_{q_3}$};
    \node[below] at (2,0) {$\varphi_{q_2}$};
    \node[below] at (3,0) {$\varphi_{q_1}$};
    \node[below] at (4,0) {$\varphi_{q_0}$};

    \node[left] at (0,0) {$0 = 1 - q_4$};
    \node[left] at (0,1) {$1 - q_3$};
    \node[left] at (0,2) {$1 - q_2$};
    \node[left] at (0,3) {$1 - q_1$};
    \node[left] at (0,4) {$1$};
  \end{scope}

\end{tikzpicture}

\caption{A pictorial representation of the onion decomposition of a graph (left) and the corresponding MMSE curve (right). Each $q_i$ represents the cumulative fraction of the graph included in all layers starting from the center till the $i$th layer; hence $q_4=1.$ The critical thresholds described above are denoted by $\varphi$, following the notation in our theorem statement below.}

\label{fig:onion}
\end{figure}

Our proof of Theorem \ref{thm:inf_main}, in addition to employing Bayesian statistical tools such as the planting trick \cite{achlioptas2008algorithmic}, is combinatorial in nature and involves a novel minimax argument (see \Cref{lem:minimax})---we are able to show that the $\min$ and $\max$ in the definition of $\varphi_q$ can be swapped, which is crucial in establishing our main result and to the best of our knowledge not implied by any standard minimax result. We believe that the onion decomposition, and the aforementioned minimax result we prove with it, may be of independent interest. We also remark that we give an efficient algorithm to compute the onion decomposition of any graph (see \Cref{lem:onion-universality}), and hence its approximate MMSE curve for large $n$. A pictorial depiction of this decomposition, and its relation to the MMSE, is shown in \Cref{fig:onion}.

Although we do not prove the subgraph Kahn--Kalai conjecture (also called the ``second" Kahn--Kalai conjecture in \cite{mossel2022second2}), our work gives a clear \emph{statistical meaning} to (variants of) the subgraph expectation thresholds from probabilistic combinatorics. For example, our main result implies that as long as $q_1=\frac{|J^{1}|}{|H|}=\Omega(1)$ for $J^{(1)}$ the densest subgraph of $H$, the original subgraph expectation threshold of $H$, $\max_{J \subseteq H} n^{-v(J)/|J| },$ is exactly (without any logarithmic multiplicative error!) the critical noise level $p$ at which the MMSE jumps from being trivial (equal to $1$), to a non-trivial value (i.e., a value strictly less than $1$). The location of the noise level at which the MMSE ceases to be trivial is a well-studied object in the literature known as the ``weak recovery'' threshold of $H$ in the planted subgraph model and there has been an intense study of the weak recovery thresholds for various planted models in the statistical physics, information theory and statistics communities (see e.g., \cite{coja2017information} and references therein).

Let us now define the weak recovery threshold in our context.
\begin{definition}
    For any $H=H_n$ and $\varepsilon \in (0,1)$ we define the $\varepsilon$-recovery recovery threshold $p_{\epsilon}=p_{\epsilon}(H)$ as
    \[p_{\epsilon}=p_{\epsilon}(H_n):=\sup\{p \in [0,1]: \MMSE_n(p) \leq 1-\epsilon\}. \]
    The planted subgraph model for $H$ is said to have a \emph{(sharp) weak recovery threshold} if for some $\epsilon_0>0$, for all $0<\epsilon<\epsilon_0$, if $n$ is sufficiently large, it holds that
    \[p_{\epsilon}=(1+o(1))p_{\epsilon_0},\]
    that is, the MMSE ``jumps" at the value $p_{\epsilon_0}$ from $1-o(1)$ to $1-\Omega(1)$. In that case, we call the sharp weak recovery threshold of $H=H_n$ to be  \[p_{\mathrm{WR}}=p_{\mathrm{WR}}(H)\defeq p_{\epsilon_0}.\]
\end{definition}

We state this result in the following informal corollary.

\begin{corollary}[Informal, based on \Cref{lem:onion-universality} and \Cref{cor:wek-dense-improved}] \label{cor:inf_main}
    Let $\eps > 0$, and $H = H_n$ an arbitrary weakly dense graph. Suppose that for $J^{(1)}=\arg \max_{J \subseteq H} |J|/v(J)$ it holds that $q_1=|J_1|/|H|=\Omega(1)$; a property referred to as $H$ being delocalized in \cite[Definition 4.3]{MNSSZ23}.  Then, the planted subgraph model for $H$ has a sharp weak recovery threshold. Moreover, this threshold $p_{\mathrm{WR}}(H)$ is equal to the subgraph expectation threshold of $H$, that is, \[p_{\mathrm{WR}}(H)=(1+o(1))\max_{J \subseteq H} n^{-v(J)/|J| }.\]  
\end{corollary}
 
Notably, the weak recovery threshold has also been called ``the condensation threshold'' of the model in the literature of statistical physics \cite{coja2017information}, due to an underlying I-MMSE relation linking the discontinuities of the MMSE and the first-order discontinuities of the free energy of the model. Interestingly, no such I-MMSE relation is currently known for the planted subgraph model for all $H$, despite multiple partial results for specific choices of $H$. To our knowledge, the best such general result is a one-sided I-MMSE relation proven in \cite[Lemma 5.9]{MNSSZ23}. 

\parhead{Beyond (weakly dense) graphs $H=H_n$: planting arbitrary monotone properties.} Recall that the theory of \emph{fractional} expectation thresholds has been established in \cite{FKP21} (as conjectured in \cite{Tal10}) to correctly predict the critical threshold $p$ for \emph{any} $H$ up to a multiplicative logarithmic factor. More generally, it predicts the critical threshold $p$ at which any monotone property is satisfied in an instance of $\mathrm{Bernoulli}(p)^{\otimes N}$ (again, up to a multiplicative logarithmic factor). Given this, it is natural to ask whether our result can be generalized to understand the limiting MMSE curve when planting arbitrary monotone properties in $\mathrm{Bernoulli}(p)^{\otimes N}$ noise.

We succeed in doing so, albeit at the cost of a constant multiplicative loss in the threshold. Our results hold under a worst-case prior over the planted element of the property, a standard setting known as the study of minimax estimation rates in statistics. As a special case, our result applies to the planted subgraph model for \emph{any} subgraph $H=H_n$ and yields the order of the density $p$ under which the MMSE of the planted subgraph model for $H$ crosses any specific value in $(0,1)$. We direct the reader to Section \ref{subsec:general} for details on this general case.

The proof of this result is based on a modification of the fractional expectation thresholds (analogous to how $\varphi_q$ is a modification of the subgraph expectation threshold), a strong duality trick first observed by Talagrand \cite{Tal10}, and finally a modification of the Bayesian proof of the spread lemma in \cite{MNSSZ23}. We refer the reader to \Cref{thm:gen} for a precise statement of this result.

\parhead{Related work.} The limiting MMSE of various high dimensional statistical models has been analyzed in a number of recent influential works. Notable characterizations of the limiting MMSE include the exact computation for the spiked rank-one matrix model under Gaussian noise \cite{dia2016mutual, lelarge2017fundamental} (conjectured in \cite{lesieur2015mmse}) and the random linear estimation problem with a Gaussian feature matrix \cite{montanari2006analysis,reeves2019replica} (conjectured in \cite{tanaka2002statistical, guo2005randomly}). In many cases, the approach begins with a non-rigorous yet remarkably precise calculation of the model's free energy via the ``replica" or ``cavity" methods, which also yields a conjectured exact formula for the limiting MMSE. 
In some cases, such as those mentioned above, the validity of these conjectured formulas has later been rigorously proven. However, our techniques fundamentally differ from those used here, and operate in a setting where typical assumptions in statistical physics literature do not hold. For instance, for our main result \Cref{thm:inf_main}, the planted subgraph $H=H_n$ is only assumed to be weakly dense for each $n$, and in particular no product-like structure is assumed on its entries as is often the case for statistical physics techniques. Moreover, for our general results on monotone properties (\Cref{thm:gen}), no assumption whatsoever is made on the planted discrete structure other than it being of a known cardinality.

In another related line of work, the goal is to prove properties of the limiting MMSE curve for specific planted models. Notable examples include \cite{reeves2020all} that understood the limiting MMSE curve in sparse linear regression and proved the AoN phenomenon, and \cite{ding2023planted, gaudio2025allsomethingnothingphasetransitionsplanted} that studied the MMSE for the planted matching problem. Our work again differs from these results in that it applies to a larger family of priors with much less structure---for instance, the planted matching model is a special case of our planted subgraph model with $H$ a perfect matching so \Cref{thm:gen} is directly applicable (as the perfect matching is not a weakly dense graph, we cannot apply the stronger \Cref{thm:wek_dense}). Moreover, our proof crucially differs in the proof techniques we employ to the previous works. Most of the above works are based on involved second moment methods (e.g., in 
 \cite{reeves2020all}) or algorithmic construction arguments (e.g., in \cite{ding2023planted}). Yet, in this work we analyze the limiting MMSE curves by leveraging Bayesian techniques, such as the planting trick \cite{achlioptas2008algorithmic}, alongside state-of-the-art combinatorics results and arguments using the structure of the (fractional) expectation thresholds, such as the duality trick of Talagrand \cite{Tal10}. We hope that the combinatorics techniques introduced in this work will be useful in understanding more minimax rates or MMSE curves in the high-dimensional statistics literature.

%% file: non-sparse_iz.tex
\newcommand{\NN}{\mathbb{N}}
\newcommand{\maxi}{\text{maximize}\quad}
\newcommand{\subjto}{\text{subject to}\quad}
\newcommand{\maxst}[2]{\maxi \;\,& #1 \\ \subjto &  #2}

\section{Getting started}

\label{sec:notation}

We now define several notions that will be useful in presenting and proving our main results below.

\subsection{Graph Theory} We start by introducing a few definitions from graph theory.

First, as mentioned in the introduction, in this paper we identify a graph $T$ with its edge set $T\subseteq \binom{[n]}{2}$. Moreover, its vertex set $V(T)$ is the set of all vertices in $K_n$ adjacent to at least one edge of $T$. Hence, we use the notations $|T|$ to refer to the number of edges in $T$ and $v(T)\defeq|V(T)|$ to refer to the number of vertices in $T.$ Finally, we define it's \emph{density} as $\rho(T)\defeq|T|/v(T).$

We now define the notion of a weakly dense graph.
\begin{definition}\label{dfn:weak_dense}
   Consider a sequence of graphs $H=H_n, n \in \mathbb{N}$. We say that a graph $H=H_n$ is weakly dense if 
   \[
   \liminf_n \frac{|H_n|}{v(H_n) \log v(H_n)}=+\infty.
   \]
\end{definition}We often simplify notation in this work and refer to $H=H_n$ as a graph, instead of a sequence of graphs indexed by $n$.

Now, we introduce a non-standard definition in graph theory, which we term a ``graph-cut," which can intuitively be viewed as a graph ``modulo'' a specific subset of its vertices. This definition will be useful in understanding the properties of the onion decomposition of a graph $H$ (see \Cref{sec:duality}).

\begin{definition}\label{dfn:graph_cut}
    A ``graph-cut'' on $n$ vertices is defined by a triplet $J = (V,S,E)$, where $S \subseteq V \subseteq [n],$ and
    \[ E \subseteq K_V \setminus K_S=\left\{ \{u,v\} : u,v \in V \text{ and at most one of $u,v$ is in $S$} \right\}. \]
    We further denote its number of edges $|J| \defeq |E|$, its number of vertices $v(J) \defeq |V \setminus S|$, and if $S \subsetneq V$, its density 
     \begin{align}\label{def:rho}
        \rho(J) \defeq \frac{|J|}{v(J)}=\frac{|E|}{|V \setminus S|}.
    \end{align}
    If $V\setminus S = \emptyset,$ we define $\rho(J):= +\infty,$ using the convention that $\frac{1}{\rho(J)}=\frac{1}{\infty}=0.$
    Note that any graph can trivially be considered a graph-cut by choosing $S = \emptyset$.
\end{definition}

Towards succinctness, given a graph $T_1$ and subgraph $T_2 \subseteq T_1$, we use $T_2|T_1$ to denote the graph-cut $(V(T_1),V(T_2),T_1\setminus T_2)$. In particular,
\begin{align}\label{def:rho_2}
    \rho(T_1|T_2) \defeq \frac{|T_1|-|T_2|}{|V(T_1) \setminus V(T_2)|}.
\end{align}





\subsection{Monotone Properties and Expectation Thresholds }\label{sec:exp_prelims}
Now we turn to some required background from combinatorics. Let $\calX$ be a finite universe of points. We start with defining a monotone property.

\begin{definition}
  A monotone property $\tilde{A}$ in the finite universe $\mathcal{X}$ is a collection of subsets of $\mathcal{X},$ such that for any $A \in \tilde{A}$ if $A \subseteq B$ then $B \in \tilde{A}.$ An element $A \in \tilde{A}$ is called a minimal element of $\tilde{A}$ if there is no $A' \in \tilde{A}$ with $A' \subsetneq A.$
\end{definition}
Observe that a monotone property $\tilde{A}$ is fully characterized by the set of minimal elements of $\tilde{A}$. We can define the fractional expectation threshold of a monotone property as in \cite{Tal10}.

\begin{definition}[Fractional expectation threshold]\label{dfn:frac_exp}
Fix any monotone property $\tilde{A}$ in a finite universe $\mathcal{X}$. We define
\[
p_{\mathrm{FE}}(\tilde{A})\defeq  \max_{w \in W} \left\lbrace \calE_w^{-1}(1/2) \right\rbrace,
\]
where $\calE_w(p)\defeq\sum_{T \subseteq \mathcal{X} } w(T)p^{|T|},$
and $W$ is the set of all ``fractional covers'' $w$ which are functions $w: 2^{\mathcal{X}} \rightarrow [0,1]$ such that
for any minimal element $A$ of $\tilde{A}$, $\sum_{T \subseteq A} w(T) \geq 1$.

\end{definition}

It turns out that that the fractional expectation threshold is a natural lower bound for the critical probability $p_c(\tilde{A})$ (via a relatively simple ``fractional" union bound argument) for the property of interest to be satisfied under the product $\Ber(p)$ measure on $\mathcal{X}$ \cite{Tal10}. Importantly, because of its fractional nature, a strong duality trick observed by Talagrand in \cite{Tal10} leads to a representation of $p_{\mathrm{FE}}(\tilde{A})$ in terms of so-called ``spread" measures on the minimal elements of $\tilde{A}$ \cite{alweiss2020improved}. Then, the breakthrough work \cite{FKP21} used this representation and proved that  $p_{\mathrm{FE}}(\tilde{A})$ also upper bounds $p_c(\tilde{A})$, up to a logarithmic multiplicative factor. As the critical probability $p_c(\tilde{A})$ is not the focus of this work, we refer the interested reader to \cite{FKP21} for more details. We only highlight that the same duality trick is in fact crucial in our work in order to obtain our general connection between minimax rates and variants of the fractional expectation thresholds in \Cref{sec:general}.

\section{Main Results I: Planting a Weakly Dense $H$}
\label{sec:non-sparse}



\subsection{Characterizing the MMSE curve via the $\varphi_q$ thresholds}


We start with our first and main result in this section. For any $q \in [0,1)$, we identify the noise level $p$ (up to $1+o(1)$ error) at which the MMSE curve crosses the threshold $1-q$. Recall the definition of $\varphi_q$ from \eqref{eq:phi-q}.

\begin{theorem}\label{thm:wek_dense}
   Consider any (sequence of) weakly dense $H=H_n$ and $q \in [0,1)$. Then, for the planted subgraph model corresponding to $H$, the following holds as $n$ grows to infinity.
  \begin{itemize}
        \item[(a)] If $\liminf_n p/\varphi_q > 1 $, then 
        \[
       \liminf_n \MMSE_n(p) \geq 1-q.
       \]
       \item[(b)] If $\limsup_n p/\varphi_q <1$, then
       \[
      \limsup_n \MMSE_n(p) \leq 1-q.
        \]
   \end{itemize}
   In particular, for any constant $q \in [0,1],$ there exists some sequence of noise levels $p=p_n$ satisfying \[\lim_n p/\varphi_q=1\] for which it holds \[\lim_n \MMSE_n(p)=1-q.\] 
\end{theorem}

The proof of the theorem is deferred to \Cref{sec:main_steps}.

\begin{remark}[The MMSE curve using the discontinuities of $\varphi_q$]\label{rem:MMSE}

\Cref{thm:wek_dense} can be used to extract non-trivial structure for the $\MMSE_n(p)$ curve from the behavior of the points $(\varphi_q)_{q \in [0,1)}$.
Observe that $\varphi_q$ is a piecewise constant function of $q,$ since for any $\epsilon=\epsilon(H) >0$ small enough, $\varphi_q = \varphi_{q+\eps}$.
For example, this holds if $\lfloor (q+\epsilon) |H| \rfloor =\lfloor q |H|\rfloor$ since in both these cases the same collection of subsets $S \subseteq H$ is considered inside the minimum operation in the definition of $\varphi_q$.

This implies that one can choose the minimal possible integer $0 \leq M \leq |H|$ for which there exist points $1=q_M>q_{M-1}>\ldots>q_1>q_0=0$ such that $\varphi_q$ is constant on the intervals $[q_{i+1},q_{i}), i=0,1,\ldots,M-1$. In terms of the MMSE, \Cref{thm:wek_dense} implies that if for some $i=0,1,\ldots,M-1$ we have $\varphi_{q_{i}}=(1+\Omega(1))\varphi_{q_{i+1}}$, then for any $\delta>0$ if $p$ satisfies
\[(1+\delta)\varphi_{q_{i+1}} \leq p \leq (1-\delta)\varphi_{q_{i}}\]
then
\[ \MMSE_n(p)=1-q_{i+1}+o(1), \]
and if $p \ge (1+\delta) \varphi_{q_0}$, $\MMSE_n(p) = 1-o(1)$.

Indeed, if $p \leq (1-\delta)\varphi_{q_{i}}$ then for any $\epsilon>0$ we have $p \leq (1-\delta)\varphi_{q_{i+1}-\epsilon}$ and therefore $\MMSE_n(p) \leq 1-q_{i+1}+\epsilon+o(1)$ by \Cref{thm:wek_dense}(b). Since $\epsilon>0$ can be made arbitrarly small, we conclude $\MMSE_n(p) \leq 1-q_{i+1}+o(1)$. For the other direction, we have $(1+\delta)\varphi_{q_{i+1}} \leq p$ and thus $\MMSE_n(p) \geq 1-q_{i+1}-o(1)$ again by \Cref{thm:wek_dense}(a) and the fact that $\MMSE_n(p)$ is a non-increasing function of $p$. The result follows. The proof when $p \ge (1+\delta) \varphi_{q_0}$ is analogous.



\end{remark}

Based on \Cref{rem:MMSE}, to further understand the MMSE curve one needs to understand the discontinuity points of the piecewise constant function, $(\varphi_q)$. At first sight, $\varphi_q$ is a minimax problem over exponentially many subgraphs, suggesting that such an analysis would be very challenging. Despite this, our next result is a full combinatorial characterization of $\varphi_q$ based on a peeling process which we term the \textit{onion decomposition} of $H$.

\subsection{Simplifying the $\varphi_q$ thresholds via the onion decomposition of $H$}
We start with formally defining the onion decomposition here. Recall from \Cref{dfn:graph_cut} that $\rho(J_1|J_2) \defeq \frac{|J_1\setminus J_2|}{|V(J_1)\setminus V(J_2)|}$ for any graphs $J_2\subseteq J_1.$ 

\begin{definition}[Onion decomposition of $H$]\label{dfn:onion}
  Let $H=H_n$ be an arbitrary graph $H=H_n$. Consider the following procedure generating an increasing sequence of subgraphs $J^{(0)}, J^{(1)},\cdots$ of $H$.
  
  \begin{enumerate}[label=(\roman*)]
      \item Set $J^{(0)} = \emptyset$.
      \item For each $t>0,$ let $J^{(t)}$ be a maximal subgraph in $\arg\max_{J \supsetneq J^{(t-1)}}\rho(J | J^{(t-1)})$\footnote{Maximality here means that there is no $J' \supsetneq J^{(t)}$ such that $J'$ also maximizes the density $\rho(J | J^{(t-1)})$.}.
      \item If $J^{(t)} = H$, stop.
  \end{enumerate}
  Let $M=M(H) \leq |H|$ be the number of steps required until the above process terminates. We refer to the outcome of this process, $J^{(t)},t=0,1,2,\ldots,M$, as the \emph{onion decomposition} of $H$.
\end{definition}
In words, the decomposition proceeds by setting $J^{(1)}$ as the densest subgraph of $H,$ then removes $J^{(1)}$ and denotes by $J^{(2)}$ the new densest subgraph, and so on. 

\begin{remark}Observe for intuition that if the graph is balanced \cite{ER60}, that is $\max_{J \subseteq H} |J|/v(H)=|H|/V(H),$ then $M=1$ and $J^{(0)}=\emptyset, J^{(1)}=H.$
\end{remark}

\begin{remark}
    Note that we refer to \Cref{dfn:onion} as \emph{the} onion decomposition of $H$. This is because, as stated in \Cref{lem:onion-universality} below, the choice of $J^{(t)}$ turns out to be unique for all $t.$
\end{remark}
The following key combinatorial lemma establishes an important relation between the onion decomposition and the function $(\varphi_q)_{q \in [0,1]}$.

\begin{restatable}{theorem}{onionuniversality}
    \label{lem:onion-universality}
    For any graph $H=H_n$, let $J^{(t)},t=0,1,2,\ldots,M$ be its onion decomposition from \Cref{dfn:onion}. The following holds. 
    \begin{enumerate}[label=(\alph*)]
        \item For any $t=0,1,2,\ldots,M-1$, given $J^{(t)}$, the choice of $J^{(t+1)}$ is unique in step (ii) of the procedure. That is, the onion decomposition of a given graph $H$ is unique.
        \item For any $q \in [0,1)$, let $t=t(q)$ be such that $|J^{(t)}|/|H| \le q<|J^{(t+1)}|/|H|$. Then 
        \[
        \varphi_q=n^{-\frac{|V(J^{(t+1)}) \setminus V(J^{(t)})|}{|J^{(t+1)} \setminus J^{(t)}|}} = n^{-\frac{1}{\rho(J^{(t+1)}|J^{(t)})}},
        \]
        i.e., the pair $(S=J^{(t)}, J=J^{(t+1)})$ is a minimax optimal pair in the definition of $\varphi_q$ in \eqref{eq:phi-q}.
        \item The onion decomposition of $H$, and in particular, the function $(\varphi_q)_{q \in [0,1]}$ can be computed in time which is polynomial in $|H|$.
    \end{enumerate}
\end{restatable}

The proof of these parts are based on a novel minimax duality principle that could be of independent interest. We defer the proof to \Cref{sec:duality}.

\begin{remark}(On the computability of $\varphi_q$.) Recall that $\varphi_q$ is defined using the minimax optimization problem where both the minimum and the maximum operation is defined over exponentially many subsets of $H$. Fortunately, the simple structural characterization of \Cref{lem:onion-universality} (a) and (b) allows us to reduce the problem to a modification of the densest subgraph problem, for which a polynomial time algorithm is known due to \cite{Cha00}. The fact that we are working with graph-cuts introduces some complications, see \Cref{sec:charikar} for details.

\end{remark}

Now notice that \Cref{rem:MMSE} characterizes the MMSE using $\varphi_q,$ and \Cref{lem:onion-universality} characterizes $\varphi_q$ using the onion decomposition of $H$. Hence, their combination allows us to obtain the following corollary directly, which characterizes the MMSE at a much finer detail than before.
\begin{corollary}\label{cor:wek-dense-improved}
    For any weakly dense graph $H=H_n$ let $J^{(t)},t=0,1,2,\ldots,M$ be its onion decomposition from \Cref{dfn:onion}. If for some $t=0,1,2,\ldots,M-1$ it holds that 
    \begin{align}\label{eq:jump}
    \frac{|V(J^{(t+1)}) \setminus V(J^{(t)})|}{|J^{(t+1)} \setminus J^{(t)}|}-\frac{|V(J^{(t)}) \setminus V(J^{(t-1)})|}{|J^{(t)} \setminus J^{(t-1)}|}=\Omega\left(\frac{1}{\log n}\right),
    \end{align} 
    then for any $\delta>0$ if 
    \begin{align}\label{eq:interval}
    (1+\delta)n^{-\frac{|V(J^{(t+1)}) \setminus V(J^{(t)})|}{|J^{(t+1)} \setminus J^{(t)}|}} \leq p \leq (1-\delta)n^{-\frac{|V(J^{(t)} \setminus V(J^{(t-1)})|}{|J^{(t)} \setminus J^{(t-1)}|}},
    \end{align}
    where in \eqref{eq:interval} we treat for convenience $n^{-0/0}\defeq1$, it holds that
    \[ 
    \MMSE_n(p)=1-|J^{(t)}|/|H|+o(1).
    \]
    Moreover, for any $\delta>0$ if $p \leq (1-\delta)n^{-\frac{|V(J^{(M)} \setminus V(J^{(M-1)})|}{|J^{(M)} \setminus J^{(M-1)}|}}$ then \[ \MMSE_n(p)=1-|J^{(M)}|/|H|+o(1)=o(1).\]
\end{corollary}The proof is deferred to \Cref{sec:ommit}. Some remarks are in order.
\begin{remark}(On the ``jump" assumption \eqref{eq:jump}.)\label{rem:jump}
     Assumption \eqref{eq:jump} asks that at the $(t+1)$th step of the onion decomposition, the maximum density $|J^{(t)}|/|V(J^{(t)})|$ reached has a non-trivial ``jump'' from the $t$th step. The role of the assumption is to ensure that the interval for $p$ considered in \eqref{eq:interval} is non-vacuous. Besides being of a technical nature, we highlight that the assumption is in fact satisfied for many $H$ of interest: for example, it can be straightforwardly checked to be satisfied for all steps $t$ for any graph $H$ with size sub-logarithmic, i.e., satisfying $|H|=o(\log n)$ (also see \Cref{sec:examples_wek} for more examples of such $H$.)

    The reason we require the assumption is that, in our main result \Cref{thm:wek_dense}, a small $1+\Omega(1)$ room between $p$ and $\varphi_q$ is needed to derive the asymptotic bounds on $\MMSE_n(p)$. We believe studying the behavior of the MMSE for $p$ being $(1+o(1))$-close to the thresholds $\varphi_q$, i.e., understanding the ``sharpness" of these transitions \cite{friedgut1999sharp,gamarnik2023sharp}, to be an interesting direction for future work.
\end{remark}

\begin{remark}(Interpretation of \Cref{cor:wek-dense-improved} and the subgraph expectation threshold) For this remark, assume a weakly dense $H$ such that the assumption \eqref{eq:jump} holds for all steps $t$ of the onion decomposition. Then, we highlight two interesting implications of \Cref{cor:wek-dense-improved} in terms of the different values the MMSE can take and the thresholds that are the MMSE discontinuity points.

First, notice that \Cref{cor:wek-dense-improved} tells us that for any noise level $p$ satisfying \eqref{eq:interval}, the Bayes-optimal performance is achieved by an estimator that recovers exactly the edges of the $J^{(t)}$ subgraph in $H^*$, up to an additive $o(1)$ error. One can therefore interpret our result as saying that for all noise levels $p$, the maximal possible subgraphs of $H$ that can be recovered is always a prescribed element of the onion decomposition of $H$. As the noise level $p$ decreases from one to zero, the first subgraph that can be recovered should be the densest one, i.e., $J^{(1)},$ the second should be the densest subgraph containing $J^{(1)},$ i.e., $J^{(2)},$ and so on, until we can recover all of $J^{(M)}=H$. It is arguably surprising that these critical subgraphs for recovery are so well structured for any weakly dense $H$.

Moreover, \Cref{cor:wek-dense-improved} implies that the thresholds at which the MMSE ``jumps'' are exactly given as functions of the onion decomposition via the formula 
\[
p=n^{-\frac{|V(J^{(t+1)} \setminus V(J^{(t)})|}{|J^{(t+1)}\setminus J^{(t)}|}} = n^{-\frac{1}{\rho(J^{(t+1)}|J^{(t)}  )}}, t=1,\ldots,M.
\] 
In particular, assuming the mild condition $|J^{(1)}|/|H|=\Omega(1)$ on $H$---a small variant of this natural condition is referred to as begin ``delocalized" in \cite{MNSSZ23}---the above observation implies that for $t=1$ the MMSE jumps from being trivial, i.e. taking the value $1-o(1),$ to its first non-trivial value $1-|J^{(1)}|/|H|+o(1)=1-\Omega(1)$ exactly at the threshold \[p=n^{-\frac{|V(J^{(1)}|}{|J^{(1)}|}}=\max_{J \subseteq H} n^{-\frac{|V(J)|}{|J|}},\] which is the subgraph expectation threshold for $H$ defined by Kahn and Kalai in \cite{KK07}, thus proving \Cref{cor:inf_main}. In other words, as mentioned in the introduction, the weak recovery threshold equals to the subgraph expectation threshold of $H$ for any such weakly dense graph $H$. In particular, our work gives a statistical meaning to the \emph{exact} subgraph expectation threshold of $H$ from probabilistic combinatorics in Bayesian statistics, without the need to account for any additional $O(\log |H|)$-factor.
\end{remark}

\begin{remark}[A characterization of the AoN phenomenon]

Recall that AoN happens for the planted subgraph model for some $H=H_n$ if the $\MMSE_n(p)$ curve converges to a step function from $0$ to $1$ as $n$ grows to infinity. An immediate consequence of \Cref{thm:wek_dense} is that the AoN phenomenon happens for a weakly dense graph if and only if for every $q,q' \in (0,1)$,
\[
\lim_n \frac{\varphi_q}{\varphi_{q'}} = 1.
\]
Interestingly, this implication can be seen as a generalization of one of the main results in \cite{MNSSZ23}, which requires an assumption of $H$ being ``delocalized'' (see \cite[Definition 4.3]{MNSSZ23}) to obtain a similar such AoN characterization.

\end{remark}

\subsection{On the convergence of the MMSE curve}
Observe that all the above results discuss the behavior of the MMSE curve when $n$ is finite but large. This is slightly diverging compared to the study of the limiting MMSE of other high dimensional statistical models, like compressed sensing \cite{reeves2019replica} or the planted clique problem \cite{MNSSZ23}, where the MMSE curve is (pointwise) converging as $n$ grows to infinity to a specific curve and the curve is characterized, often via methods from statistical physics. This discrepancy is fundamental in our result as we consider a much larger family of priors compared to the above results. Furthermore, we do not assume any joint structure on $H_n$ across different values of $n$; in other words, the planted subgraph $H_n$, and hence our prior, may change arbitrarily with $n$. For example, $H$ could be a $10\log n$-clique for even $n$, and a $\log \log n$-clique for odd $n$, leading to two convergent MMSE subsequences with different limits. More specifically, in this case, if say $p=1/2$, it can be easily proven via the AoN theory from \cite{MNSSZ23} that the (normalized) MMSE equals $o(1)$ for even $n$ and equals to $1-o(1)$ for odd $n$, hence the MMSE as a curve is \emph{not converging} as $n$ grows to infinty. Yet, under the assumption that the sequence of graphs $H_n$ is such that the pointwise limit of the MMSE \emph{does} exist, we have the following asymptotic formula as a direct corollary of \Cref{thm:wek_dense}.

\begin{restatable}{corollary}{limitpoint}
   \label{cor:limit-point-inversion}
   Let $H=H_n$ be a weakly dense graph and $p=p_n \in (0,1).$ Let the function $G: (0,1] \rightarrow [0,+\infty)$ defined by $G(q)=\varphi_q.$ Consider the sequence $\alpha_n \in (0,1)$ defined by \[\alpha_n =\inf G^{-1}([0,p]), n \in \mathbb{N}. \]
   Then any limit point $x \in [0,1]$ of $\{\alpha_n : n\in \mathbb{N}\},$  corresponds in an one-to-one fashion to a limit point of the sequence $\MMSE_n(p).$ In particular, if the $\lim_n \MMSE_n(p)$ exists, then the $\lim_n \alpha_n$ exists and it holds,
   
   \begin{align}\label{eq:mmse_fin}
        \lim_n \MMSE_n(p)=1-\lim_n \alpha_n.
    \end{align} 
\end{restatable}

The proof is based on simple calculus and is deferred to \Cref{sec:ommit}.

\subsection{A simple example}\label{sec:examples_wek}
We now work through a simple example to apply our established MMSE theory on weakly dense $H=H_n$ that illustrates the diversity of possible MMSE curves that one can obtain as the planted graph $H=H_n$ is varied. We emphasize at the outset that, while the simple example below can be solved directly by other methods, the calculation illustrates how our main theorem gives a general recipe that can be applied to compute the MMSE curve of any planted graph which is weakly dense.

\parhead{A union of disjoint cliques.} 
Let $m=m_n$ and $r_1> \cdots> r_m \geq 1,$ be an arbitrary sequence of integers with \[1/r_{i+1}-1/r_{i} \geq \delta /\log n\] for some constant $\delta>0.$ Such sequences include the cases $r_i=(C_i+o(1)) \log n, i=1,2,\ldots,m$ for some constants $C_1>C_2>\ldots>C_m>0$, or $r_i, i=1,\ldots,n$ is an arbitrary decreasing sequence of positive integers with $r_1=o(\sqrt{\log n}).$

We then consider $H$ to be the \emph{disjoint} union of $m$ cliques of sizes $r_1> \cdots> r_m \geq 1,$ so that \[H = \bigcup_{s\leq m}K_{r_s}\] and $k\defeq |H| = \sum_{s\leq m}\binom{r_s}{2}$.

In terms of the onion decomposition of $H$ (\Cref{dfn:onion}), we have $J^{(0)}=\emptyset$ and the densest subgraph of $H$ can be straightforwardly checked to be $J^{(1)}=K_{r_1}.$ Similarly, a simple inductive argument implies that $M=M(H)=m$ and \[J^{(t)} = \bigcup_{s\leq t}K_{r_s}, t=1,2,\ldots,m.\] Moreover, from \Cref{lem:onion-universality}, we know that for any $1>q \geq 0$, if $m-1 \geq t \geq 0$ is maximal such that $|J^{(t)}|/|H| \leq q$, it holds 
\[
\varphi_q=p_t \defeq n^{-\frac{|V(J^{(t+1)})\setminus V(J^{(t)})|}{|E(J^{(t+1)})\setminus E(J^{(t)})|}} = n^{-\frac{2}{r_{t+1} - 1}}.
\]We also set $p_m\defeq\varphi_1=0$ and $p_{-1}\defeq1.$
Then using \Cref{cor:wek-dense-improved} we have that for any $\delta>0$ and for each $-1\leq t\leq m-1,$ if $(1+\delta)p_{t+1} \leq p \leq (1-\delta)p_t,$ then

\[ \MMSE_n(p)=1-\frac{\sum_{s\leq t+1}\binom{r_s}{2}}{k}+o(1).\]

For example, if $r_i=(C_i+o(1)) \log n, i=1,2,\ldots,m$ for some constants $C_1>C_2>\ldots>C_m>0$ we conclude that $p_t=e^{-2/C_{t+1}}+o(1)$ for all $0 \leq t \leq m-1$ and $p_m=0, p_{-1}=0.$ Hence, we have that for any $\delta>0$ and for each $-1\leq t\leq m-1,$ if \[(1+\delta)e^{-2/C_{t+2}} \leq p \leq (1-\delta)e^{-2/C_{t+1}},\] then it holds
\[ \MMSE_n(p)=1-\frac{\sum_{s\leq t+1}C^2_s}{\sum_{s\leq m}C^2_s}+o(1),\]where set $e^{-2/C_{m+1}}\defeq0,e^{-2/C_0}\defeq1$. 

If $m=1$ and $C_1=2/\log 2$, we recover the AoN phenomenon for the planted clique model  from  \cite{MNSSZ23} where the  the clique is of size $k=(2+o(1))\log_2 n$ and the critical AoN noise level is $p=e^{-2/C_1}=1/2$.

\section{Main Results II: Planting general monotone properties}

\label{subsec:general}

We now turn to a more general model, referred to as the planted subset model in \cite{MNSSZ23}. In this general setting, we provide a weaker characterization of the MMSE curve as $n$ grows, but under \emph{no} assumptions on the planted structure. In particular, for the planted subgraph model discussed previously, this more general result applies to the MMSE curve for any $H=H_n$, even if not weakly dense.

To define this, let $N \in \mathbb{N}$ and fix some arbitrary family of $k=k_N$-sets $\calA=\{A_i\}_{i=1}^{M}$ in a finite universe $\mathcal{X}$ of $N$ elements, and some noise level $p=p_N$. Now,  the ``signal'' is an element $A$ from $\calA$ which is sampled from an (arbitrary but known) prior $\pi$ on $\calA$. Then, the statistician observes \[Y=A \cup X_p\] where the ``noise" $X_p$ is a sample from the $p$-biased product measure on $\mathcal{X}$, i.e., $X_p \sim \Ber(p)^{\otimes \mathcal{X}},$ and aims to recover $A$ from $Y$.

\begin{remark}One may want to imagine $\calA$ as the set of minimal elements of some monotone property of interest (see \Cref{sec:exp_prelims} for the standard definitions on monotone properties). Then the statistician aims to recover the planted ``certificate'' of this property in the presence of Bernoulli noise. The question of interest is: for what noise levels $p$ is it possible to infer the certificate of the monotone property?

For example, the planted subgraph model is a special case of the planted subset model. Consider the universe $\mathcal{X}$ to correspond to the $N=\binom{n}{2}$ edges of $K_n,$ the complete graph on $n$ vertices. In that case, the Bernoulli noise $X_p$ equals in distribution to $G(n,p)$. Now, for any subgraph $H=H_n$ of $K_n$, let $\calA$ to be the collection of all subsets of $\mathcal{X}$ (i.e., all subgraphs of $K_n$) which are copies of $H$ in $K_n.$ Under this choice of $\mathcal{X}$ and $\mathcal{A}$, the planted subset model becomes the planted subgraph model (it is easy to prove that the ``worst'' prior for planted subgraph models is the uniform one, see \Cref{lem:unif-is-worst-case} for the details).

Of course, the planted subset model generalizes well beyond the planted subgraph model to other interesting models, such as planting a connectivity certificate (i.e., a spanning tree) in $G(n,p)$ or planting arbitrary monotone properties on random hypergraphs.
\end{remark}

As in the previous section, we focus on the MMSE of the model as a function of $p$, defined as
\[ \MMSE_N(p,\pi)=\frac{1}{k}\min_{\hat{A}} \E_{A \sim \pi,Y=A\cup X_p,\hat{A}} \left[ \|\mathbf{1}_A-\hat{A}(Y)\|^2_2 \right]=\frac{1}{k}\E_{A \sim \pi,Y=A\cup X_p} \left[ \|\mathbf{1}_A-\E[\mathbf{1}_A|Y]\|^2_2 \right], \]
where the minimum is over all randomized estimators $\hat{A}: 2^{\mathcal{X}} \rightarrow \mathbb{R}^N$. As the prior is arbitrary in this model, we focus on the MMSE of estimating $A$ under the ``worst'' possible prior $\pi$, which using the strong duality of linear programs can be written as

\begin{align}\label{eq:minimax}
\mathrm{MinMax}(\calA,p)\defeq\max_{\pi} \MMSE_N(p,\pi)
&=\frac{1}{k}\min_{\hat{A}}\max_{\pi}\E_{A,Y,\hat{A}} \left[ \|\hat{A}(Y)-\mathbf{1}_A\|^2_2 \right].
\end{align}

\begin{remark}
    We note that the study of minimax rates of estimation, like $\mathrm{MinMax}(\calA,p)$, where one accounts for the worst-case prior is a topic of intense study for many decades in the literature of theoretical statistics. Moreover, the optimizers $\pi, \hat{A}$ in the definition of $\eqref{eq:minimax}$ are called the ``least favorable prior" and the ``minimax estimator", respectively.
\end{remark}
Our main result is to identify the order of the noise level $p$ for which $\mathrm{MinMax}(\calA,p)$ achieves a given fixed value of interest in $(0,1]$. Recall that in the planted subgraph model for a weakly dense $H$ and under a uniform prior, \Cref{thm:wek_dense} implies that $\MMSE_n(p)=1-q+o(1)$ for some $p=(1+o(1))\varphi_q$ where the threshold $\varphi_q$ is a modified subgraph expectation threshold. In the planted subset model, for any $q \in [0,1),$ the desired threshold is given by a modification of the so-called \emph{fractional expectation} thresholds of the monotone property generated by $\mathcal{A}$.

\begin{definition}
Fix any family of $k$-sets $\mathcal{A}$ in a finite universe $\mathcal{X}$ of $N$ elements. For any $q \in [0,1),$ we define
\[
\psi_{q}\defeq(\psi_q)_N=\min_{S\subseteq \mathcal{X}, |S|\leq qk} \left\lbrace \max_{w \in W_{S}} \left\lbrace \calE_w^{-1}(1/2) \right\rbrace \right\rbrace,
\]
where given $S \subset \mathcal{X}$, $\calE_w(p)\defeq\sum_{T \subseteq \mathcal{X} } w(T)p^{|T \setminus S|},$
and $W_S$ is the set of all ``fractional covers'' $w$ which are functions $w: 2^{\mathcal{X}} \rightarrow [0,1]$  satisfying the following two properties.
\begin{enumerate}[label=(\alph*)]
   \item (``Support condition") Unless $|T| \geq qk$ and $S \subseteq T$, it holds $w(T)=0$.

    \item (``Fractional cover condition") For all $A_i, i=1,2,\ldots,M$ where $S \subseteq A_i$ we have \[\sum_{T \subseteq A_i} w(T) \geq 1.\]
\end{enumerate}
We finally set $\psi_1\defeq0$.
\end{definition}

\begin{remark}
We remark that the case $q=0$ and $S=\emptyset$ corresponds exactly to the fractional expectation threshold \cite{Tal10} of the monotone property with minimal elements $\mathcal{A}$, see \Cref{dfn:frac_exp}. In particular, the value of the fractional expectation threshold has been calculated up to constants for a number of monotone properties, including the $H$-inclusion property for a subgraph $H=H_n$ of the complete graph $K_n$ \cite{FKP21,mossel2022second}. 
\end{remark}

\subsection{Characterizing the Minimax curve via the $\psi_q$ thresholds}
We now present our general result in this setting.

\begin{theorem}\label{thm:gen}
    Let $q \in [0,1)$, and $\calA$ be an arbitrary $k$-uniform family of subsets of $[N]$. Then, as $N \to \infty$,
    \begin{itemize}
        \item[(a)] If $\lim_N p/\psi_q =+\infty,$ then 
        \[
        \liminf_N \mathrm{MinMax}(\mathcal{A},p) \geq 1-q.
        \]
        \item[(b)] If $\lim_N p/\psi_q =0,$ then
        \[
        \limsup_N \mathrm{MinMax}(\mathcal{A},p) \leq 1-q.
        \]
    \end{itemize}
    
\end{theorem}

The proof is deferred to \Cref{sec:general}.

\begin{remark}
    We remark that while we are able to efficiently compute $\varphi_q$ in the setting of weakly dense graphs using the onion decomposition of $H$, we do not yet have a way to efficiently compute an optimizer for $\psi_q$ for arbitrary monotone properties. It is an interesting question for future work if a similar onion-peeling process exists for any monotone property.

    Despite this, given the success of probabilistic combinatorics in computing the fractional expectation thresholds for various different properties \cite{FKP21}, we find the calculation of $\psi_q$ for specific monotone properties an interesting direction---by \Cref{thm:gen}, this would have implications for the minimax rates of such problems. We next go over such a calculation for the well-studied case where $H$ is a perfect matching.
\end{remark}

\subsection{ A simple example}\label{sec:gen_example}

We now work through a simple example in this more general case.

\parhead{Perfect matching.}
Let us consider the case that $H$ is the perfect matching on $n$ vertices ($n$ even). This is a model that has received significant interest in the community (see e.g., \cite{ding2023planted,gaudio2025allsomethingnothingphasetransitionsplanted} and references therein). In this setting, the minimax rate is equal to the MMSE for the uniform distribution over planted matchings (see \Cref{lem:unif-is-worst-case}). 

We will show via direct computation that for any $q\in(0,1)$, \[\psi_q=(1+o(1))\frac{e}{n}.\] Applying \Cref{thm:gen}, we can conclude the following. \begin{itemize}
    \item If $p=o\lprp{\frac{1}{n}}$, then the MMSE will be $o(1)$ for any prior. We note that this is a result already established in \cite{ding2023planted}, by a direct analysis in this setting. 
    \item If $p=\omega\lprp{\frac{1}{n}}$, then there exists a prior (the uniform distribution) for which the MMSE will be $1-o(1)$. We note that this was very recently independently proved by \cite{gaudio2025allsomethingnothingphasetransitionsplanted}, using a direct analysis.
\end{itemize}
Recall that
\[
\psi_{q} = \min_{S\subseteq \mathcal{X}, |S|\leq qk} \left\lbrace \max_{w \in W_{S}} \left\lbrace \calE_w^{-1}(1/2) \right\rbrace \right\rbrace.
\]

First, we observe that for any choice of S, the max is the fractional Kahn--Kalai threshold $\tau_k$ for a perfect matching on $k=n-v(S)$ vertices. This is because any $S$ will be disconnected from the rest of $H$, so the resulting graph-cut is essentially just a normal graph.
Therefore, we have that
\begin{align*}
    \psi_q 
    &= \min_{S\subseteq \mathcal{X}, |S|\leq qk} \left\lbrace \tau_{n-v(S)} \right\rbrace \\
    &= \min_{k\geq (1-q)n} \tau_{k}
\end{align*}

It remains to compute $\tau_k$, the fractional Kahn Kalai threshold for a perfect matching in a size $k$ graph.

From \cite{mossel2022second}, for subgraph inclusion properties, the fractional Kahn--Kalai threshold is equal to the following:
$$\tau_n = \max_{J\subseteq H} \lprp{\frac{H[J]}{K_n[J]}}^{1/|J|}$$
where $G[J]$ for a graph $G$ refers to the number of unlabeled copies of $J$ in $G$.

In the case of perfect matchings, $\tau_k$ can be expressed as
\[ \tau_k = \max_{0 \le r \le k/2} \left( \frac{ M_{k/2}[M_r] }{ K_k[M_r] } \right)^{2/r} = \max_{1 \le r \le k/2} \left( \frac{ \binom{k/2}{r} }{ \binom{k}{2r} \cdot (2r-1)!! } \right)^{2/r}, \]
with $M_t$ denoting the matching with $t$ edges. We shall show that the above is maximized for $r = k/2$. Letting $a_r$ be the expression inside the parentheses, we have for $r \ge 1$,
\[ a_r = \frac{ \binom{k/2}{r} }{ \binom{k}{2r} \cdot (2r-1)!! } = \frac{ \binom{k/2}{r-1}  }{ \binom{k}{2r-2} \cdot (2r-3)!! } \cdot \frac{ \frac{k/2 - r}{r} }{ \frac{(k-2r)(k-2r+1)}{2r(2r-1)} \cdot (2r-1) } = a_{r-1} \cdot \frac{1}{k-2r+1}. \]
A simple inductive argument then implies that $a_r \le \left(\frac{1}{k-2r-1}\right)^r$. Indeed, this is trivially true for $r=1$. For larger $r$, we have by induction that $a_r = \frac{1}{k-2r+1} \cdot a_{r-1} \le \frac{1}{k-2r+1} \cdot \left( \frac{1}{k-2r+1} \right)^{r-1} \le \left(\frac{1}{k-2r-1}\right)^r$ as desired. Given this,
\[ a_r^{1/r} = \left( \frac{1}{k-2r+1} \cdot a_{r-1} \right)^{1/r} \ge \left( f(r-1)^{1/r-1} \cdot a_{r-1} \right)^{1/r} = a_{r-1}^{1/(r-1)}. \]
It follows that the maximum is attained for $r = k/2$. Thus,
\begin{align*}
    \tau_k
    &=  \lprp{\frac{1}{(k-1)!!}}^{2/k}
\end{align*}
Observe that this is minimized at $k=n$. Approximating $(n-1)!! = \lprp{\frac{n}{e}}^{n/2+o(1)}$, we have that 

\begin{align*}
    \tau_n
    &=\lprp{(e/n)^{n/2+o(1)}}^{2/n}\\
    &= (1+o(1))\frac{e}{n}.
\end{align*}
Hence we conclude that 
\[
\psi_q = (1+o(1))\frac{e}{n}
\]
for all $q\in [0,1]$, as desired.

%% file: appendix_iz.tex
\section{Onion Decomposition: Proof of \Cref{lem:onion-universality}}

\label{sec:duality}

In this section, we establish some key results on the onion decomposition that will be useful in multiple proofs throughout the paper and also prove \Cref{lem:onion-universality}.

We start with recalling the $\rho$ notation from \Cref{sec:notation}. Notice that our key quantity $(\varphi_q)$ can be defined via $\rho$; if we define for $q \in [0,1]$ the quantity
    \begin{align}\label{eq:rho_q}
 \rho_{q} \defeq \min\left\{ \max\left\{ \rho(J | S) : J \supseteq S \right\} : S \subseteq H, |S| \le q|H| \right\} 
\end{align}
    then it clearly holds
\[ \varphi_q=n^{-\frac{1}{\rho_q}}.\]

\subsection{Key lemmas} We now establish four useful combinatorial lemmas for the ``density" $\rho$ function, that follow from elementary graph theory arguments. We defer their proofs to \Cref{sec:ommit}.

\begin{restatable}{lemma}{densitymodularity}
    \label{lem:density-inequality-modularity}
    Let $A,B$ be subgraphs. Then,
    \[ 
    \rho(A\cup B|A) \ge \rho(B|A \cap B).
    \]
\end{restatable}


\begin{restatable}{lemma}{lembalance}\label{lem:balance}
    \label{lem:balanced-density-inequality}
    Let $K|T$ be a graph-cut that is ``balanced'', in that for all $T'$ such that $T \subseteq T' \subseteq K$, $\rho(T'|T) \le \rho(K|T)$. Then, for all $T'$  such that $T \subseteq T' \subseteq K$ it also holds \[\rho(K|T') \ge \rho(K|T).\]
\end{restatable}

The third lemma implies that modulo any subgraph $A,$ there exists a unique maximizer of the density $\rho(J|A), J \supseteq A.$

\begin{restatable}{lemma}{maximizerunion}\label{lem:maximizer-union}
    
    Let $A$ be a subgraph and consider the map $\varphi_A$ which maps any subgraph $J$ containing $A$ to $\varphi_A(J)=\rho(J|A).$ If $J'$ and $J''$ are two subgraphs containing $A$ and at least one of them is balanced, then $\varphi_A(J'\cup J'') \geq \min\{\varphi_A(J'), \varphi_A(J''')\}$. In particular, if $J', J''$ are two maximizers of $\varphi_A$, then $J'\cup J''$ is also a maximizer of $\varphi_A$, so there is a unique maximal maximizer of $\varphi_A.$
\end{restatable}

The fourth lemma asserts that the densities in the onion decomposition are non-increasing.

\begin{restatable}{lemma}{union}\label{lem:monoton_onion}
    Let $A$ be a subgraph and $B=\arg\max_{J \supseteq A} \rho (J|A).$ Then for all $C \supseteq B$, \[\rho(C|B) \leq \rho(B|A).\] In particular, for $J^{(t)}, t=0,1,\ldots,M$ the onion decomposition of any subgraph $H$ the sequence $\rho(J^{(t+1)}|J^{(t)}), t=0,1,\ldots,M-1$ is non-increasing.
\end{restatable}

\subsection{Simplifying $\varphi_q$: Proofs of \Cref{lem:onion-universality}(a) and (b)}
The above lemmas allow us to conclude parts (a) and (b) of \Cref{lem:onion-universality}, restated for convenience.

\onionuniversality*

\Cref{lem:maximizer-union} applied to $A=J^{(t)}$ immediately implies (a). 

\begin{proof}[Proof of (b)] Fix some $q \in [0,1)$ and let $(S^*,J^*)$ be a minimax pair in the definition of $\rho_q$, which in particular also satisfies
\[\varphi_q=n^{-1/\rho(J^*|S^*)}.\]
In particular, our goal is to show that for $t$ such that $|J^{(t)}| \leq q |H|<|J^{(t+1)}|$, it holds that $\rho(J^{(t+1)} | J^{(t)}) = \rho(J^* | S^*)$.

We start with choosing $t$ to be the step of the onion decomposition such that $J^{(t)}\subseteq S^*$ but $J^{(t+1)}\not\subseteq S^*$. We first prove that for this choice of $t$,
 \begin{align} \label{eq:exists_t}
\rho(J^{(t+1)} | J^{(t)}) = \rho(J^* | S^*)
\end{align}
By the definition of the onion decomposition, $J^{(t+1)} \in \argmax_{J \supseteq J^{(t)}} \rho(J | J^{(t)})$. Combining this with the optimality of $S^*$, that means we have
\[ \rho\left( J^{(t+1)} | J^{(t)} \right) = \max_{J \supseteq J^{(t)}} \rho(J|J^{(t)}) \ge \rho(J^* | S^*). \]
Next, defining $\wt{J} = J^{(t+1)} \cup S^*$ and $\wt{S} = J^{(t+1)} \cap S^*$, \Cref{lem:density-inequality-modularity} implies that
\[ \rho(\wt{J} | S^*) \ge \rho(J^{(t+1)} | \wt{S}). \]
Finally, by \Cref{lem:balanced-density-inequality},
\[ \rho(J^{(t+1)} | \wt{S}) \ge \rho(J^{(t+1)} | J^{(t)}). \]Hence, combining the last three displayed inequalities, we have \[\rho(\wt{J} | S^*) \geq \rho(J^* | S^*).\] But by the optimality of $J^*$ given $S^*$,  $\max_{J \supseteq S^*} \rho(J|S^*)= \rho(J^* | S^*)$, all the above four displayed inequalities should be equalities, proving \eqref{eq:exists_t}.

Thus, for all $q \in [0,1]$, \[\rho_q \in \{ \rho(J^{(t+1)} | J^{(t)}) : 0 \leq t \leq q |H|\}.\] But again as by definition for all $t$, $\rho\left( J^{(t+1)} | J^{(t)} \right) = \max_{J \supseteq J^{(t)}} \rho(J|J^{(t)}) $ it also holds by setting $S=J^{(t)}, t \leq q|H|$ and the definition of $\rho_q$ that \[ \rho_q \leq \min \{ \rho(J^{(t+1)} | J^{(t)}) : 0 \leq t \leq q |H|\}.\]Hence, \[\rho_q = \min \{ \rho(J^{(t+1)} | J^{(t)}) : 0 \leq t \leq q |H|\}.\] But by \Cref{lem:monoton_onion} we have that $\rho(J^{(t+1)} | J^{(t)})$ is a non-increasing function of $t$. Therefore the last displayed equation gives $\rho_q=\rho(J^{(t+1)} | J^{(t)})$ for the maximum possible $t \leq q |H|$, from which (b) follows.
\end{proof}

\subsection{Computability of $\varphi_q$: Proof of \Cref{lem:onion-universality}(c)}\label{sec:charikar}

We further have the following lemma, which shows that this procedure can be implemented in polynomial-time.

\begin{proof}
    We first show that there is an efficient algorithm to find the densest subgraph in a graph-cut. Specifically, given a graph-cut $G|S$, we show how to find a $J\subseteq G|S$ attaining the value $f(G|S) = \max_{J\subseteq G|S}\rho(J)$.
    
    The algorithm and its analysis are based on Charikar's Linear Programming (LP) approach to find a densest subgraph in a graph \cite{Cha00}. We consider the following LP:
    \begin{align*}
        \maxst{\sum_{ij \in G\setminus S} x_{ij}}
        {
        x_{ij} \leq y_i \qquad \qquad \qquad \forall i\in V(G)\setminus V(S) \\
        &x_{ij} \leq y_i \qquad \qquad \qquad\forall j\in V(G)\setminus V(S) \\
        &\sum_{i \in V(G)\setminus V(S)} y_i \leq 1 \\
        &x_{ij}, y_i \geq 0 \qquad \qquad \qquad \forall ij \in G|S.
        }
    \end{align*}
    Let $v^*$ be the optimal value of this LP. To see that $v^*\geq f(G|S),$ note that for every $J\subseteq G|S,$ letting $y_i = \frac{1}{v(J)}\bone\{i \in V(J)\}$ and $x_{ij} = \frac{1}{v(J)}\bone\{ij \in J\}$, we get $\sum_{ij \in G|S} x_{ij} =\rho(J)$. Moreover, $\sum_{i\in V(G)\setminus V(S)} y_i = 1$, and the other two constraints are trivially satisfied.
    
    Next, given a solution to the LP with value $v^*,$ we show how to extract a choice of $J$ such that $\rho(J)\geq v^*.$ Let $\{\overline{x}_{ij}\}_{ij}\cup \{\overline{y}_i\}_i$ be such a solution. Note that we can assume without loss of generality that $\overline{x}_{ij} = \min\{\overline{y}_i, \overline{y}_j\}$ when $i,j \in V(G)\setminus V(S)$, $\overline{x}_{ij} = \overline{y}_i$ when $j\in V(S)$, and $\overline{x}_{ij} = \overline{y}_j$ when $i\in V(S).$  For each $r \in \R_{\geq 0},$ consider the subgraph $E(r) = \{ij : \overline{x}_{ij}\geq r\} \subseteq G \setminus S$ and denote by $V(r) \defeq\{i : \overline{y}_i\geq r\}$. Note that for each $r$, by our assumption, for any edge of the subgraph $E(r)$, $ij\in E(r)$ it can be easily checked to hold $\{i,j\}\setminus V(S) \subseteq V(r)$, and therefore
    \begin{align}\label{eq:inclusion}
        V(E(r))\setminus V(S) \subseteq V(r).
    \end{align}Now, we claim that there exists $r$ such that $\rho(V(r), E(r)) \geq v^*.$ Indeed, since $\int_0^\infty |E(r)|dr = \sum_{ij \in G|S}\overline{x}_{ij}$ and $\int_0^\infty |V(r)|dr = \sum_{i \in V(G)\setminus V(S))}\overline{y}_{i}\leq 1$, we get
    \begin{align*}
        v^* &\leq \frac{\int_0^\infty |E(r)|dr}{\int_0^\infty |V(r)|dr} \\
        &\leq \sup_r \frac{|E(r)|}{|V(r)|}\\
        & \leq \sup_r \frac{|E(r) \setminus S|}{|V(E(r))\setminus V(S)|}.
    \end{align*}
    But then note that we can efficiently enumerate all combinatorially distinct $ E(r)$ to find a $J=E(r)$ such that
    \[\rho(J)=\rho(E(r))=\frac{|E(r) \setminus S|}{|V(E(r))\setminus V(S)|} \geq v^*,\]
    as desired.
    Given the above, we can efficiently construct sequences $\emptyset=\wt{J}^0,\wt{J}^1,\dots, \wt{J}^\ell$ and $\wt{J}^{(t)}=\cup_{s\leq t}\wt{J}^s$ so that, for each $t\leq \ell-1$, $\wt{J}^{t+1}$ is a maximizer of $J\mapsto \rho(J|\wt{J}^{(t-1)})$ among $J \supseteq \wt{J}^{(t-1)}.$
    
    To efficiently construct the onion decomposition, we need yet to show something slightly stronger: we need to efficiently construct the sequence $\emptyset = J^0, J^1,\dots, J^{m}$ such that, for $J^{(t)}=\cup_{s\leq t} J^s$, we have that $J^{t+1}$ is the \emph{maximal} $J$ maximizing the function $J\mapsto \rho(J| \wt{J}^{(t-1)})$ among $J \supseteq \wt{J}^{(t-1)}$ (the uniqueness of the maximal $J$ follows by \Cref{lem:onion-universality}).

    We do this in a greedy fashion as follows. Let $s_1\geq 1$ be maximal such that \[\rho(\wt{J}^1| \wt{J}^{(0)}) =\cdots = \rho(\wt{J}^{s_1}|\wt{J}^{(s_1-1)}),\] and let $J^1 = \cup_{s\leq s_1}\wt{J}^{s}.$ Then let $s_2$ be maximal such that \[\rho(\wt{J}^{s_1+1}|\wt{J}^{(s_1)}) =\cdots = \rho(\wt{J}^{s_2}| \wt{J}^{(s_2-1)}),\] and let $J^2 = \cup_{s_1<s\leq s_2}\wt{J}^s,$ and so on. Notice that using \Cref{lem:monoton_onion} we have for all $t,$
    \begin{align}\label{eq:step_down}
       \rho(\wt{J}^{s_t+1}|\wt{J}^{(s_t)})<\rho(\wt{J}^{s_t}|\wt{J}^{(s_t-1)}).
    \end{align}
    
    Now, it is straightforward to check that for all $t$ $J^t$ maximizes $J\mapsto \rho (J| J^{(t-1)}).$ Hence, it remains to show that $J^t$ is in fact a \emph{maximal} such maximize for all. Suppose that this is not the case for some $t$. Then by \cref{lem:maximizer-union}, there exists $J'\not\subseteq J^t$ so that $J'\cup J^t$ is the maximal maximizer of $J\mapsto \rho (J| J^{(t-1)}).$ In particular, it holds \[\rho(J'\cup J^t|J^{(t-1)})=\rho(J^t|J^{(t-1)})=\rho(\wt{J}^{s_t}|\wt{J}^{s_{t-1}})\] which implies for $ J''\defeq J'\setminus J^{t}$ that 
    \[\rho(J''|\wt{J}^{s_{t}})=\rho(\wt{J}^{s_t}|\wt{J}^{s_{t-1}}).\]We can therefore guarantee that
    \[\rho(\wt{J}^{s_t+1}|\wt{J}^{(s_t)})\geq \rho(J''|\wt{J}^{s_{t}})=\rho(\wt{J}^{s_t}|\wt{J}^{(s_t-1)}),\]a contradiction with \eqref{eq:step_down}.
\end{proof}


\section{The MMSE for weakly dense $H$: Proof of \Cref{thm:wek_dense}}
\label{sec:main_steps}

In this section, we give an overview of the main technical steps involved in the proof of our main result, \Cref{thm:wek_dense}. Our first step is a ``positive'' result, showing that if $p$ is smaller than a threshold $\wt{\varphi}_q,$ then the noise is small enough so that the limiting MMSE is at least $1-q$. The second step is a ``negative'' result, proving that if $p$ is larger than $\varphi_q$ then the noise is high enough so that the limiting MMSE is at most $1-q$. Our final step shows that these two thresholds, $\wt{\varphi}_q$ and $\varphi_q$, are in fact equal.

\subsection{Two lemmas: the importance of the order of the min and max operators}

Our first lemma, proved in \Cref{sec:ease}, shows that for $p$ below a modified version of the $\varphi_q$ threshold defined in \eqref{eq:phi-q} (where importantly the min and max operators have been exchanged) it is possible to achieve MMSE at most $1-q$.

\begin{restatable}[Possibility of Recovery]{lemma}{easyrecovery}\label{lem:ease}
  Suppose $H=H_n$ is a weakly dense $H$.  Given $q\in [0,1)$, define the threshold
    \[
         \wt{\varphi}_{q} = (\wt{\varphi}_{q})_n=\max\left\{ \min\left\{ n^{-\frac{|V(J)\setminus V(S)|}{|J\setminus S|}} :  S \subseteq J, |S| \le       q |H|\right\} : J \subseteq H \right\}. 
    \]
    Then for any $p=p_n$ with $ \limsup_n p/ \wt{\varphi}_{q} <1,$ it holds
    \begin{align*}
        \limsup_n \MMSE_n(p) \leq 1-q.
    \end{align*}
\end{restatable}
The lemma is proven by showing that for $p$ below this new $\wt{\varphi}_{q}$ threshold, with high probability there does not exist an isomorphic copy $H'$ of the planted graph $H^*$ in the observed graph $G=H^* \cup G_0, G_0 \sim G(n,p)$ which overlaps with $H^*$ in fewer than $(q-o(1))|H|$ edges. We then prove via the Nishimori identity that the MMSE is equal to one minus the expected overlap between a random isomorphic copy of $H$ and $H^*,$ from which we can directly upper bound the MMSE by $1-q+o(1).$ The proof is deferred to \Cref{sec:ease}.

The intuition behind the variational formula in \Cref{lem:ease} for $\wt{\varphi}_q$ is as follows. In order to avoid having a copy $H'$ of $H$ in $G$ with overlap $\leq q|H|$ with the planted $H^*,$ we need to understand for which noise levels $p$ we can rule out such structures. Let us denote by $S$ this potential intersection set of edges between a copy $H'$ of $H,$ and $H^*.$ Once we fix such an $S,$ a bound on the largest value of $p$ such that no copy of $H$ appears in the observed graph $G$ overlapping with $H^*$ exactly on $S$ can be obtained by a standard first moment method (i.e, a vanilla union bound). Then optimizing the bound to account for the worst-case $S,$ leads to a bound on $p$ to have MMSE at most $1-q.$ However, in many cases this bound will not be tight. To improve it, using the intuition from the tightness of the expectation thresholds from \cite{KK07}, one can fix a subgraph $J\subseteq H$ and instead upper bound for any $p$ the probability that a copy $H'$ of $H$ in $G$ overlaps $H^*$ on $S$ by the probability that \emph{a copy of $J$} in $G$ overlaps with $H^*$ in at most $\leq q |H|$ edges (simply because $J \subseteq H$). Then, the same standard first moment method argument as above can be performed with $J$ in place of $H$ to obtain a new upper bound on the probability that a copy $H'$ of $H$ in $G$ overlaps $H^*$ on $S$, which depends on $J$ and $S$. Now, notice that by optimizing in a worst-case sense over $S$ (to guarantee an MMSE at most $1-q$), and then optimizing over all subgraphs $J$ (now, in favor of the statistician so choosing the ``best-case" $J$), we arrive exactly at the max-min formula for $\wt{\varphi}_q$ giving \Cref{lem:ease}. 

We next describe our second lemma, proved in \Cref{sec:difficulty}, which shows that above the $\varphi_q$ threshold, the MMSE must be worse than $1-q$.

\begin{restatable}[Impossibility of recovery]{lemma}{difficultrecovery}\label{lem:difficulty}
   Suppose $H=H_n$ is a weakly dense $H$. Given $q\in [0,1)$, recall the threshold $\varphi_q$ from \eqref{eq:phi-q}.
    Then for any $p=p_n$ with $ \liminf_{n} p /\varphi_q>1$, it holds
    \begin{align*}
        \liminf_n \MMSE_n(p) \geq 1-q.
    \end{align*}
\end{restatable}

The proof shows that when $p$ crosses $\varphi_q$, the posterior distribution of $H^*$, which is the uniform distribution over the copies of $H$ in the observed graph, is concentrated on copies with overlap less than $q|H|$ with $H^*.$ Our proof follows from an appropriate leverage of the planting trick \cite{achlioptas2008algorithmic} similar to its application in the recent Bayesian proof of the spread lemma \cite{mossel2022second}. The proof is deferred to \Cref{sec:difficulty}.

Here, the intuition for $\varphi_q$ is as follows. Given $S \subseteq H$ with $|S| \leq q |H|$, one can define a statistically easier version of the estimation task of interest, in which the statistician is also given as side information the $S$ part of the planted graph $H^*.$  An application of the data processing inequality implies that the inability to recover a given fraction of $H^*$ in this easier problem implies the same inability in the original statistical problem (Such techniques are often called ``genie" lower bound arguments, e.g, \cite{reeves2013approximate}).

To prove the MMSE lower bound it suffices to determine the noise levels $p$ for which the MMSE is at least $1-q$ when the statistician is given $S$ as side information, and then optimizing over $S$ to find the minimum such noise level. For this reason, we start by fixing any $S$ with $|S| \leq q |H|$ and assume that the $S$ part of $H$ is given to the statistician. Then a simple posterior calculation implies that MMSE is at least $1-|S|/|H| \geq 1-q$, if a typical copy $H'$ (of $H$) in $G = H^* \cup G_0$ that contains $S$ does not overlap an $\Omega(1)$-fraction of edges in $H^* \setminus S$. Equivalently, this is true if most copies $H'$ containing $S$ are such that $H' \setminus S$ is (almost) fully included in $G_0$.

For this reason, it makes sense to study the threshold of $p$ at which many (near-complete) copies of $H\setminus S$ start appearing in $G_0 \sim G(n,p)$. The intuition from the literature on Kahn--Kalai thresholds suggests that this threshold is given by the subgraph expectation threshold of $H \setminus S$, which is precisely the maximizing problem in the definition of $\varphi_q$.
While the tightness of this threshold is still at a conjectural phase and expected to hold only up to a logarithmic error, we prove that the approximate version of it (recall we want $H \setminus S$ to only approximately appear in $G$) that holds exactly above the subgraph expectation threshold. Optimizing over all subgraphs $S$ to determine the best such bound results in $\varphi_q$, the minimum possible subgraph expectation threshold of $H \setminus S$ among all $S$.

\subsection{Exchanging the max and min operators}

To conclude \Cref{thm:wek_dense}, we must relate the thresholds in the two lemmas above. It is somewhat dissatisfying that the natural probabilistic way to argue about \Cref{lem:ease} and \Cref{lem:difficulty} leads to almost the same threshold for $p$, but with the max and the min operators being exchanged in the two results. While there are plenty of ``minimax results" in the literature for exchanging the min and the max operators (e.g., in the literature of game theory and of optimization), no such result appears applicable in our discrete vertex-symmetric setting. Yet, a careful combinatorial argument allows us to indeed exchange the operators. This minimax result is the content of the next lemma, proved in \Cref{sec:dual_pf}, and is proved by employing the onion decomposition of $H$.

\begin{restatable}[Minimax result]{lemma}{minimax} \label{lem:minimax}
    For any $n$, $q=q_n \in [0,1)$, and $H=H_n$, $\wt{\varphi}_q = \varphi_q$.
\end{restatable}

We remark that the duality lemma above is an exact, non-asymptotic equality: it holds for all graphs $H=H_n$ and for all $n$. In order to prove it, our main tool is a structural result for the max-min and min-max optimal pairs $(S,J)$ in $\varphi_q$. Observe that the pair $(S,J)$ can take exponentially many values. Yet, by considering the onion decomposition $(J^{(t)})$ of $H$ (\Cref{dfn:onion}), it turns out that, as explained in \Cref{sec:duality}, the minimax optimization defining $\varphi_q$ has a simple optimizing pair: it is exactly of the form $(S = J^{(t)}, J = J^{(t+1)})$, where $t$ is maximal such that $|J^{(t)}|\leq q|H|$. For instance, for $q$ close to 1, the optimal pair is to choose $S=J^{(1)}$, the densest subgraph of $H$ and $J=J^{(2)}$ the ``second'' densest subgraph in the onion decomposition. The knowledge of the optimizers greatly simplifies $\varphi_q$ and is key to establishing \Cref{lem:minimax}.

\subsection{Putting it all together}
Notice that the above lemmas immediately imply \Cref{thm:wek_dense}.
\begin{proof}[Proof of \Cref{thm:wek_dense}]
    Parts (a) and (b) follow immediately from \Cref{lem:ease,lem:difficulty,lem:minimax}. 
    
    The last part follows from the fact that $\MMSE_n(p)$ is a non-increasing polynomial function $p$ (see \Cref{lem:MMSE_poly}) together with a simple analysis argument to construct the desired sequence. Indeed, for any $\epsilon>0,$ and for large enough $n$, parts (a), (b) imply
    \[\MMSE((1-\epsilon)\varphi_q) \leq 1-q+\epsilon\] and \[\MMSE((1+\epsilon)\varphi_q) \geq 1-q-\epsilon.\] That means that for any $\epsilon>0$ there exists $n_0=n_0(\epsilon)$ large enough such that for all $n \geq n_0$ there exists $p=p_n \in [(1-\epsilon)\varphi_q,(1+\epsilon)\varphi_q]$ for which \[1-q-\epsilon \leq \MMSE(p) \leq 1-q+\epsilon.\]
    Iterating this argument for $\epsilon=1/m,m=1,2,\ldots$ leads to (potentially distinct) sequences of noise levels $p^{(m)}=p^{(m)}_n$ such that for all $m$, there exists a (without loss of generality strictly increasing) $G(m) \in \mathbb{N}$ such that for all $n \geq G(m),$ $p^{(m)}_n \in [(1-1/m)\varphi_q,(1+1/m)\varphi_q]$ and \[1-q-1/m \leq \MMSE(p^{(m)}_n) \leq 1-q+1/m.\]Now, ``diagonalizing" this argument, consider the sequence defined as follows: for each $n \in [G(m),G(m+1))$, let \[p=p_n\defeq p^{(m)}_{n}\] (for $n<G(1),$ set $p_n=0$). Then notice that clearly for this sequence it holds $\lim_n p_n/\varphi_q =1$ and also $\lim_n \MMSE_n(p_n)=1-q$ as desired.
\end{proof}

\subsection{Proof of \Cref{lem:ease}: Recovery for $p \le (1-\eps) \wt{\varphi}_{q}$}

\label{sec:ease}


In this section, we prove \Cref{lem:ease}. In particular, we assume that for some $\varepsilon>0$ and $n$ large enough it holds $p \leq (1-\varepsilon)\varphi_q$.

\begin{proof}
    Fix $\delta > 0$ to be a constant and set $k\defeq|H|$. It suffices to show that the $\MMSE_n(p)$ is at most \[1-q+\delta+o(1).\]  Since $\delta$ is arbitrary, we may eventually take $\delta$ to $0$ to obtain the desired statement.

    Observe that by Bayes' rule, the posterior of the model is simply the uniform distribution over all copies of $H$ in the observed graph. Using the Nishimori identity (\Cref{lem:Nishimori}), we have
    \[ 1-\MMSE_n(p) = \frac{1}{k} \E_{\substack{\bG \sim G(n,p) \\ H^* \sim \text{uniform copy of $H$ in $K_n$}}}\left[ \E_{H' \sim \text{copies of $H$ in $\bG \cup H^*$}} |H' \cap H^*| \right]. \]
    We bound this as follows,
    \begin{align*}
        1-\MMSE_n(p) &\ge (q-\delta) - \frac{1}{k} \E_{\substack{\bG \sim G(n,p) \\ H^* \sim \text{uniform copy of $H$ in $K_n$}}}\left[ \E_{H' \sim \text{copies of $H$ in $\bG \cup H^*$}} |H' \cap H^*| \cdot \bone_{|H' \cap H^*| \le (q-\delta)k} \right] \\
            &\ge (q-\delta) - \E_{\substack{\bG \sim G(n,p) \\ H^* \sim \text{uniform copy of $H$ in $K_n$}}}\left[ \E_{H' \sim \text{copies of $H$ in $\bG \cup H^*$}} \bone_{|H' \cap H^*| \le (q-\delta)k} \right] \\
            &\ge (q-\delta) - \mathbb{P}_{\bG,H^*}\left[ \text{$\exists$ a copy $H'$ of $H$ in $G$ such that $|H^* \cap H'| \le (q-\delta)k$} \right]
    \end{align*}
    It suffices to show that  \[\mathbb{P}_{\bG,H^*}\left[ \text{$\exists$ a copy $H'$ of $H$ in $G$ such that $|H^* \cap H'| \le (q-\delta)k$} \right]=o(1).\] For starters, let $J^*$ be an optimal choice of $J$ in the definition of $\wt{\varphi}_q$. Then, using the notation $\cong$ to denote the graph isomorphism relation, we have
    \begin{multline*}
        \mathbb{P} \left[ \text{$\exists$ $H' \subseteq G \cup H^*$ such that $H' \cong H$ and $|H^* \cap H'| \le (q-\delta)k$} \right] \\ \le \mathbb{P} \left[ \text{$\exists$ $J' \subseteq G \cup H^*$ such that $J' \cong J^*$ and $|H^* \cap J'| \le (q-\delta)k$} \right]
    \end{multline*}
    This can be bounded further by Markov's inequality and the definition of the $G(n,p)$ noise to give,
    \begin{align*}
        &\mathbb{P} \left[ \text{$\exists$ $J' \subseteq G \cup H^*$ such that $J' \cong J^*$ and $|H^* \cap J'| \le (q-\delta)k$} \right]\\
        &\qquad\qquad\qquad\qquad\le \sum_{\ell \le (q-\delta)k} \E[ \left| \text{$J' \subseteq G \cup H^*$ such that $J' \cong J^*$ and $|H^* \cap J'| = \ell$} \right| ] \\
        &\qquad\qquad\qquad\qquad\le \sum_{\ell \le (q-\delta)k} \mathbb{P}_{J',H' \sim K_n} \left[ |J' \cap H'| = \ell \right] \cdot M_J p^{L-\ell}, 
    \end{align*}
    where $M_J$ is the number of copies of $J$ in $K_n$, and $L = |J| \ge qk$ (by \Cref{lem:onion-universality}(b)) and by $\mathbb{P}_{J',H' \sim K_n}$ we refer to the joint measure of $J'$, and copy of $J^*$ in $K_n$ and $H'$ an independent copy of $H$ in $K_n$. 
    
    By employing \cite[Lemma 4.4]{MNSSZ23},
    \[ \mathbb{P}_{J',H' \sim K_n} \left[ |J' \cap H'| = \ell \right] \le v(H)^{O(v(H))} \cdot n^{-v_\ell(J)}, \]
    where $v_\ell(J)$ is the minimal number of vertices in a subgraph of $J$ with $\ell$ edges.
    Therefore,
    \[ \MMSE_n(p) - (1-q+\delta) \le  v(H)^{O(v(H))} \sum_{\ell \le (q-\delta)k} \cdot n^{-v_\ell(J)} \cdot M_J p^{L-\ell} \]
    Now observing that $M_J = (n)_{v(J)}/|\mathrm{Aut}(H)| \leq n^{v(J)}$, we get that
    \[ \MMSE_n(p) - (1-q+\delta) \le \sum_{\ell \le (q-\delta)k} v(H)^{O(v(H))} \cdot n^{v(J) - v_\ell(J)} \cdot p^{L-\ell}. \]
    By the definitions of $\wt{\varphi}_{q}$ and $J^*$, we have that for all $\ell \leq (q-\delta)k$,
    \[ \wt{\varphi}_{q} \le n^{- \frac{v(J) - v_\ell(J)}{L-\ell}}. \]
    Thus, for large enough $n$, since $p \leq (1-\epsilon)\wt{\varphi}_{q},$ we also have 
    \[ \MMSE_n(p) - (1-q+\delta) \le  v(H)^{O(v(H))} \sum_{\ell \le (q-\delta)k}  (1-\eps)^{L-\ell} \le  v(H)^{O(v(H))} k (1-\eps)^{\delta k} = o(1), \]
    as desired.

\end{proof}

\subsection{Proof of \Cref{lem:difficulty}: Impossibility of recovery for $p\ge (1+\eps)\varphi_q$}

\label{sec:difficulty}

In this section, we prove \Cref{lem:difficulty}. The proof of this lemma will essentially be a formalization of the intuition mentioned above.

By definition of $\varphi_q$, there exists some $S \subseteq H$ with $|S| \le qk$, for which we have 
\begin{align*}
    p &\geq (1+\eps) \cdot \max\left\{ n^{-\frac{|V(T) \setminus V(S)|}{|T \setminus S|}} : T \subseteq H \right\}.
\end{align*}
Let $\delta > 0$ be a fixed but arbitrary small constant, which we will eventually shrink it to $0$. In particular, it must hold
\begin{align*} 
    p &\geq (1+\eps) \cdot \max\left\{ n^{-\frac{|V(T) \setminus V(S)|}{|T \setminus S|}} : T \subseteq H , |T\setminus S|\geq \delta k.\right\}.
\end{align*}
Our goal will be to show that, for all $p$ satisfying the inequality above, it will be impossible to recover better than a $\delta$-fraction of the set of edges $S^c.$ This will imply by standard Bayesian tools that
\[
\MMSE_n(p)\geq 1-q - O(\delta)-o(1),
\]
and thus \Cref{lem:difficulty}, since $\delta$ can be taken arbitrarily small.

More concretely, our technique will be to show that even if the statistician is given the $S^*$-part of the planted graph $H^*$ as side information, it will be impossible for the statistician to recover an $\Omega(1)$ fraction of the rest of $H^* \setminus S^*$; in particular,  with side information the statistician must miss at least a $1-q-o(1)$ fraction of $H*$, and then via a data processing argument, we can conclude the same for the original setting without side information.





To define this alternate problem, consider the following inference problem defined on slightly modified graphs where the edges of a particular subgraph of $H$ is removed from the observations.

\begin{definition}[Planted $S^c$-graph]\label{defn:planted-graph-cut}
    Let $H=H_n$ be a graph and $H^*$ be a uniformly random draw from the isomorphism class of $H$ in $K_n$. Let also $S$ be a subgraph of $H$ and $S^*$ an arbitrary copy of $S$ in $H^*$. The \emph{planted $S^c$-graph} problem is the task to infer $H^* \setminus S^*$ given a draw from $G=(H^* \setminus S^*) \cup G_0$, where $G_0$ is the product Bernoulli measure with parameter $p$ only on the edges of $K_n \setminus S^*$. Moreover, we define the MMSE of the problem,
\[ \MMSE_{n,S}(p)=\frac{1}{|H|- |S|}\E[ \|1_{H^* \setminus S^*}-\E[1_{H^* \setminus S^*}\|^2_2 |G]. \]
\end{definition}
Notice that in the planted $S^c$ graph problem, the observed $G$ is only supported on the edges of $K_n \setminus S^*$.

We start with the following reduction from the original problem to the planted graph-cut problem for $S$, at an MMSE loss of a multiplicative factor $1-q.$ 
\begin{lemma}[Reduction]
    \label{lem:reduction}
    Let $S \subseteq H$ with $|S| \leq q|H|$. Recall $\MMSE_n(p)$ is the MMSE of the original planted problem and $\MMSE_{n,S}(p)$ is the MMSE in the planted $S^c$-graph-cut problem. Then it holds \[\MMSE_{n,S}(p)(1-q) \le \MMSE_n(p).\]
\end{lemma}
\begin{proof}
    Given a copy $\wt{H}$ of $H$, let us denote by $S(\wt{H})$ the part of $\wt{H}$ corresponding to a copy of $S$ (resolving ties arbitrarily). Notice that for any fixed copy $S_0$ of $S$ in $K_n,$ by the law of total variance we have
    \begin{align*}
        \MMSE_{n}(p)=\frac{1}{|H|}\E \mathrm{Var}(H^*|G) \geq \frac{1}{|H|} \E \mathrm{Var}(H^*|G,S\wt(H)=S_0)
    \end{align*}
    But now consider the statistical model where the prior is uniform over all copies $H'$ of $H$ so that $S(H')=S_0,$ and the noise is as usual an independent instance of $G(n,p).$ The posterior $\mathbb{P}'_{S_0}$ in this model is uniform among all copies $H'$ of $H$ in the observed graph with $S(H')=S_0$, and therefore the \emph{unnormalized} MMSE in this model satisfies
    \[\MMSE'_{S_0,n}(p)=\E \mathrm{Var}(H^*|G,S\wt(H)=S_0).\]
    In particular, \[\frac{1}{|H|} \MMSE'_{S_0,n}(p) \leq \MMSE_n(p).\] Moreover clearly the posterior in the planted $S_0^c$-graph model is the marginal of $\mathbb{P}'_{S_0}$ on $K_n \setminus S_0$. Therefore by two applications of the Nishimori identity \Cref{lem:Nishimori} (identifying graphs on $n$ vertices with vectors in $\{0,1\}^{\binom{n}{2}}),$ 
    \[\MMSE'_{S_0,n}(p)=|H|-\E_{v,v' \sim^{\mathrm{i.i.d.}} \mathbb{P'}_{S_0}} \langle v,v' \rangle=|H|-(|S|+\E_{v,v' \sim^{\mathrm{i.i.d.}} \mathbb{P}_S} \langle v,v' \rangle )=(|H|-|S|)\MMSE_{S_0,n}(p) \]The second equality follows because the bits about the edges of $S_0$ are always equal to one under $\mathbb{P}'_{S_0}.$ The result follows since for any copy $S_0$ of $S$, clearly $\MMSE_{S_0,n}(p)=\MMSE_{S,n}(p),$ and of course $|H|-|S| \geq (1-q)|H|.$

\end{proof}

Given \Cref{lem:reduction}, to complete the proof of impossibility of recovery, we require the following impossibility result for the planted graph-cut problem.

\begin{restatable}{lemma}{plantedgraphcutdifficulty}
    \label{lem:post-reduction-lower-bd}
    Let $q \in [0,1),$ $H=H_n$ a weakly dense graph and $S=S_n \subseteq H, |S| \leq q |H|.$ Consider the planted $S^c$-graph problem for the planted graph $H$. Then, assume that for some constants $\epsilon, \delta>0,$ we have for large enough $n$,
    \[ 
    p \geq (1+\eps)\max\left\{ n^{-\frac{|V(T) \setminus V(S)|}{|T \setminus S|}} : T \subseteq H , |T\setminus S|\geq \delta |H|.\right\}.
    \]
    Then \[\MMSE_{S,n}(p) \geq 1-\delta/(1-q)-o(1).\]
\end{restatable}

\begin{proof}Let $k\defeq|H|-|S|$. Let also $\wt{G} =  (H^*\setminus S) \cup G_0$ be a sample from the planted $S^c$-graph distribution as defined in \Cref{defn:planted-graph-cut}. In particular, let $H^*$ the planted copy of $H$ and $S^*$ the (arbitrarily chosen) copy of $S$ in $H^*.$

First notice that it suffices to show that
\begin{align}\label{eq:goal_1}
        \P\left[\frac{|(H' \setminus S^*) \cap (H^* \setminus S^*)|}{k} \geq \delta/(1-q) \right] &= o_N(1),
    \end{align}
where $H'$ is a uniformly random isomorphic copy of $H$ containing $S^*$ that is present in the graph $\wt{G} \cup S^*.$ Indeed, a simple application of Baye's rule implies that the posterior of the planted $S^c$-graph problem for $H$ is uniform on all $H' \setminus S^*$ where $H'$ is a copy of $H$ in $\wt{G} \cup S^*$ containing $S^*$. Given that \eqref{eq:goal_1} and Nishimori's identity (\Cref{lem:Nishimori}) imply the result.
    
    In the calculation below, we always condition on $S^*.$ For convenience, we introduce some new notation. We denote by $(S^*)^c$-graph, any graph on edges only in $K_n \setminus S^*.$ We define $G'$ to be an independent sample from the product $\mathrm{Bernoulli}(p)$ measure on $K_n \setminus S^*$ (that is $G'$ has no planted copy $H^* \setminus S^*$) and $\delta'\defeq\delta/(1-q)$. 
    We also denote for an $(S^*)^c$-graph $G$,
    \begin{itemize}
        \item[(a)] by $Z(G)$ the number of isomorphic copies $H'$ of $H $ containing $S^*$ in $G \cup S^*,$
        \item[(b)] by $Z_\ell(H^*,G)$ the number of copies $H'$ of $H $ containing $S^*$ in $G \cup S^*$ with \[|(H' \setminus S^*) \cap (H^* \setminus S^*)|=\ell, \]
        and
        \item[(c)] by $Z^2(\ell,G)$ is the number of pairs of copies $H',H''$ of $H$ containing $S^*$ in $G \cup S^*$ with  \[|(H' \setminus S^*) \cap (H'' \setminus S^*)|=\ell.\]
    \end{itemize}

For any $\eps>0$ We have that by direct calculations (similar to the proof of \cite[Lemma 3.11]{MNSSZ23} where the universe are all edges minus the edges in $S$)
    \begin{align*}
        \P\left[\frac{|(H' \setminus S^*) \cap (H^* \setminus S^*)|}{k} \geq \delta' \right] &= \E\left[\frac{1}{Z(\wt{G})}\sum_{\ell\geq k\delta'} Z_\ell(H', \wt{G})\right] \\
          &\leq \P[Z(\wt{G})\leq \eps \E Z(G')] + \frac{1}{\eps \E Z(G')}\E \sum_{\ell \geq k\delta'}Z_\ell(H^*, \wt{G}) \\
          &\leq \eps + \frac{1}{\eps \E Z(G')}\E \sum_{\ell \geq k\delta'}Z_\ell(H^*, \wt{G}) \\
          &\leq \eps + \frac{1}{\eps (\E Z(G'))^2}\E \sum_{\ell \geq k\delta'}Z^2(\ell, G') \\
          &= \eps + \frac{1}{\eps} \sum_{\ell\geq k\delta'} \frac{\mathrm{P}_{H'}(|(H' \setminus S^*) \cap (H^* \setminus S^*)| = \ell)}{p^\ell}
    \end{align*}
    where in the two last inequalities we used a Bayesian inference idea called the ``planting trick" \cite{achlioptas2008algorithmic}, as executed in \cite[Lemmas 3.9, 3.10]{MNSSZ23}, which allows us to pass from the planted observation graph $\tilde{G}$ to the ``null" graph $G'$. Moreover, in the last equality under $\mathrm{P}_{H'}$, $H'$ is a uniform copy random isomorphic copy of $H$ in $K_n$ containing $S^*$ and $H^*$ is without loss of generality any fixed arbitrary copy of $H$ in $K_n$ containing $S*$. Hence it suffices to show that \[\sum_{\ell\geq k\delta'} \frac{\mathrm{P}_{H'}(|(H' \setminus S^*) \cap (H^* \setminus S^*)| = \ell)}{p^\ell} = o_N(1).\]

Now we employ an idea \cite[Lemma 4.4]{MNSSZ23} to show this, by adjusting it to our $S^c$-graph setting. Fix some $\ell \geq \delta' k$. Notice that for a copy $H'$ to satisfy $|(H' \setminus S^*) \cap (H^* \setminus S^*)| = \ell$ it must hold $|V(H' \cap H^*)\setminus V(S^*)| \geq v_{\ell,S^*}(H)$ where  $v_{\ell,S^*}(H)$ is the minimum value of $|V(T) \setminus V(S^*)|$ among all subsets $T$ of $(H^* \setminus S^*)$ with $|T|=\ell$ edges. But to choose the random copy $H'$ of $H$ containing $S^*$,  $|V(H' \cap H^*)\setminus V(S^*)|$ is simply an Hypergeometric distribution with parameters $(v(H)-v(S^*),v(H)-v(S^*),n-v(S^*))$. Hence,
    \begin{align*}
        \mathrm{P}_{H'}(|(H' \setminus S^*) \cap (H^* \setminus S^*)|)&\leq   \binom{n-v(S^*)}{v(H)-v(S^*)}^{-1} \sum_{w= v_{\ell,S^*}(H)}^{v(H)-v(S^*)} \binom{v(H)-v(S^*)}{w}\binom{n-v(H)}{v(H) -v(S^*)-w} \\
        &\leq v(H)^{O(v(H))} \sum_{w=v_{\ell,S^*}(H)}^{v(H)-v(S^*)}\frac{(n - v(H))_{v(H) - v(S^*)-w}}{(n-v(S^*))_{v(H)-v(S^*)}} \\
        &\leq v(H)^{O(v(H))} \frac{(n - v(H))_{v(H) - v(S^*)- v_{\ell,S^*}(H)}}{(n-v(S^*))_{v(H)-v(S^*)}} \\
         &\leq v(H)^{O(v(H))} (n-v(H))^{- v_{\ell,S^*}(H)} \\
         &\leq v(H)^{O(v(H))} n^{- v_{\ell,S^*}(H)}.
    \end{align*} Notice let $T_1  \subseteq H^* \setminus S^*$ be such that $|T_1|=\ell$ and $|V(T) \setminus V(S^*)|=v_{\ell,S^*}(H)$. Hence for $T=T_1 \cup S^*$ it must hold $|T \setminus S^*|=\ell$ and $|V(T) \setminus V(S^*)|=v_{\ell,S*}(H).$ Now by assumption on $p$ since $\ell \geq \delta' k=\delta (|H|-|S|)/(1-q) \geq \delta |H|$ it must hold that $p^{|T \setminus S^*|} n^{|V(T) \setminus V(S^*)|}\geq (1+\eps)^{|T\setminus S^*|}$ and therefore \[p^{\ell} n^{v_{\ell,S^*}(H)} \geq (1+\eps)^{\ell}.\]

Combining the last two displayed equations we have for all $\ell \in [\delta' k,k],$
    \begin{align*}
        \frac{ \mathrm{P}_{H'}(|(H' \setminus S^*) \cap (H^* \setminus S^*)|)}{p{\ell}} &\leq v(H)^{O(v(H))} (1+\eps)^{-\ell},
\end{align*}hence
    \begin{align*}
        \sum_{\ell\geq k\delta'}  \frac{ \mathrm{P}_{H'}(|(H' \setminus S^*) \cap (H^* \setminus S^*)|)}{p{\ell}} &\leq \sum_{\ell\geq k\delta}  \frac{v(H)^{O(v(H))}}{(1+\eps)^\ell } \\
        &= v(H)^{O(v(H))} (1+\eps)^{-\delta' k}=o(1)
    \end{align*}
    where we have used that $\ell\geq k\delta' = \delta' (|H|-|S|)/ (1-q) \geq \delta |H| = \omega(v(H)\log v(H))$ where the last inequality is by the weakly dense assumption on $H$. This concludes the proof.
\end{proof}







\subsection{Proof of \Cref{lem:minimax}: A minimax result}\label{sec:dual_pf}
In this section, we proof \Cref{lem:minimax}, which is the last missing part of proving \Cref{thm:wek_dense}. To do so we establish the following stronger claim. 

\begin{theorem}
    \label{th:min-max-duality}
    For $q \in (0,1)$, define
    \[ \rho_{q} = \min\left\{ \max\left\{ \rho(J | S) : J \supseteq S \right\} : S \subseteq E(H), |S| \le qk \right\} \]
    and
    \[ \wt{\rho}_q = \max\left\{ \min\left\{ \rho(J|S) : S \subseteq J, |S| \le qk \right\} : J \subseteq E(H) \right\}. \]
    Then,
    \begin{enumerate}[label=(\alph*)]
        \item $\rho_q = \wt{\rho}_q$. Equivalently, $\varphi_{q} = \wt{\varphi}_{q}$.
        \item There exists an optimizer $(S^*,J^*)$ in the definition of $\rho_q$ (and $\wt{\rho}_q$) such that $|J^*| \ge qk$.
    \end{enumerate}
\end{theorem}


Clearly, \Cref{th:min-max-duality}(a) immediately yields \Cref{lem:minimax}.
\begin{proof}[Proof of \Cref{th:min-max-duality}]
    It is not difficult to see that \Cref{lem:onion-universality}(b), in conjunction with \Cref{th:min-max-duality}(a), immediately yields \Cref{th:min-max-duality}(b). It remains to prove (a).

    Let us first show the easier direction, that $\wt{\rho}_q \le \rho_q$. 
    Let $(S^*,J^*)$ be an optimizer in the definition of $\wt{\rho}_q$, and $S \subseteq H$ be arbitrary. Our goal is to show that there exists $J' \supseteq S$ such that $\rho(J' | S) \ge \rho(J^* | S^*)$. Set $J' = J^* \cup S$ and $T = J^* \cap S$. Then, by \Cref{lem:density-inequality-modularity},
    \[ \rho(J' | S) \ge \rho(J^* | T) \ge \rho(J^* | S^*), \]
    where the second inequality is because $\rho(J^* | S^*) = \min_{S \subseteq J} \rho(J^* | S)$, so $\wt{\rho}_q \le \rho_q$.

    Next, suppose that $(J^{(t)},J^{(t+1)})$ is an optimizer in the definition of $\rho_q$, using \Cref{lem:onion-universality}(b). We shall show that
    \[ \wt{\rho}_q \ge \min_{S \subseteq J^{(t+1)}} \rho(J^{(t+1)} | S) \ge \rho(J^{(t+1)} | J^{(t)}) = \rho_q. \]
    The first inequality follows from the definition of $\wt{\rho}_q$ and that $|J^{(t+1)}| > q |H|$, so it remains to show the middle inequality. Let $S \subseteq J^{(t+1)}$ be arbitrary, and for each $1 \leq i \le t+1$, set $S_i = S \cap J^{(i)}$. First off, note that $\left(\rho(J^{(r+1)} | J^{(r)})\right)_{r \ge 0}$ is a non-increasing sequence in $r$ by \Cref{lem:monoton_onion}. By \Cref{lem:balanced-density-inequality}, for all $0 \leq i \leq t$, $\rho(J^{(i+1)} | S_{i+1} \cup J^{(i)}) \ge \rho(J^{(i+1)} | J^{(i)})$. As a result, for all $i \leq t,$ \[\rho(J^{(i+1)} | S_{i+1} \cup J^{(i)}) \ge \rho(J^{(t+1)} | J^{(t)}) = \rho_q.\] But by leveraging the increasing structure of the $J^{(i)}, 0 \leq i \leq t+1$, it holds
    \begin{align*}
        \rho(J^{(t+1)} | S) &=\frac{|J^{(t+1)}\setminus S|}{|V(J^{(t+1)})|-|V(S)|}\\
&= \frac{\sum_{i=0}^{t} |(J^{(i+1)}\setminus J^{(i)}) \cap (J^{(t+1)}\setminus S)| }{\sum_{i=0}^{t} |V( J^{(i+1)})|-|V(J^{(i)} \cup S_{i+1})| } \\
&= \frac{\sum_{i=0}^{t} |(J^{(i+1)}\setminus J^{(i)}) \cap (J^{(i+1)}\setminus S)| }{\sum_{i=0}^{t} |V( J^{(i+1)})|-|V(J^{(i)} \cup S_{i+1})| } \\
&\ge \frac{\sum_{i=0}^{t} | J^{(i+1)} |-| J^{(i)} \cup S_{i+1} | }{\sum_{i=1}^{t}  |V( J^{(i+1)})|-|V(J^{(i)} \cup S_{i+1})| } \ge \min_{0 \leq i \leq t} \rho(J^{(i+1)} | S_{i+1} \cup J^{(i)})  \geq \rho_q,
    \end{align*}
    as desired.
\end{proof}

\section{The Minimax rate for any monotone property: Proof of \Cref{thm:gen}}\label{sec:general}

We now prove \Cref{thm:gen}. We do this by first establishing the positive result (part (b)), and then the negative result (part (a)).

\subsection{The positive result: Proof of \Cref{thm:gen}(b)}

The proof follows by an appropriate application of a ``fractional'' first moment method. While this method has been used in the probabilistic combinatorics literature, we are not aware of any such application in the high dimensional statistics literature.

\begin{proof}
We establish the following property. For any $\delta>0$, we have that with high probability for any $A_i \in \calA$ with $A_i \subseteq Y$ it must hold $|A_i \cap A| \geq (q-\delta) k$.

Notice that this property is sufficient to conclude the result. Indeed, Baye's rule implies that for an arbitrary prior $\pi$ on $\calA$ the posterior is assign mass to each $A_i \in \calA$ given by the formula $\hat{\pi}(A_i) \propto \pi(A_i)\mathbf{1}(A_i \subseteq Y).$ In particular, the suggested property implies that with high probability the posterior outputs with high probability an $A_i$ with $|A_i \cap A| \geq (q-\delta) k.$ Therefore by Nishimori's identity (\Cref{lem:Nishimori}),
\begin{align*}
\MMSE_N(p,\pi) \leq 1- (q-\delta)+o(1).
\end{align*} Since $\delta>0$ is arbitrary we conclude the desired result.

Now to prove the desired property, first observe that since $Y=A \cup \mathcal{X}_p$,
\begin{align}
&\mathbb{P}_{A,Y} \left[ \text{$\exists$ $A_i \in \calA$ such that $A_i \subseteq Y, |A \cap A_i| \le (q-\delta)k$} \right] \nonumber \\
=&\E_{A} \sum_{S \subseteq A, |S| \leq (q-\delta)k} \mathbb{P}_{Y} \left[ \text{$\exists$ $A_i \in \calA$ such that $A_i \subseteq Y, A \cap A_i=S$} \right] \nonumber \\
=&\E_{A} \sum_{S \subseteq A, |S| \leq (q-\delta)k} \mathbb{P}_{\mathcal{X}_p} \left[ \text{$\exists$ $A_i \in \calA$ such that $A_i \setminus S \subseteq \mathcal{X}_p, A_i \subseteq S$} \right]. \label{eq:union_1}
\end{align}
Now for any $S$ we can choose $w_S \in W_S$ with $\sum_{T} w_S(T) (\psi_q)^{|T\setminus S|} \leq 1/2$. Notice that since by definition of $w_S$ we have $\min_{A_i \in \calA, S \subseteq A_i } \sum_{T \subseteq A_i} w_S(T) \geq 1,$ it also holds
\[ \bone\left[ \text{$\exists$ $A_i \in \calA$ such that $A_i \setminus S \subseteq \mathcal{X}_p$} \right] \leq \sum_{T} w_S(T) 1(T \setminus S \subseteq \mathcal{X}_p)\] and therefore 

\[\E_{A} \sum_{S \subseteq A, |S| \leq (q-\delta)k} \mathbb{P}_{\mathcal{X}_p} \left[ \text{$\exists$ $A_i \in \calA$ such that $A_i \setminus S \subseteq \mathcal{X}_p$} \right] \leq \E_{A} \sum_{S \subseteq A, |S| \leq (q-\delta)k} \sum_{T} w_S(T) p^{|T\setminus S|}.\]Combining with \eqref{eq:union_1} gives, 
\[\mathbb{P}_{A,Y} \left[ \text{$\exists$ $A_i \in \calA$ such that $A_i \subseteq Y, |A \cap A_i| \le qk$} \right] \leq \E_{A} \sum_{S \subseteq A, |S| \leq (q-\delta)k} \sum_{T} w_S(T) p^{|T\setminus S|}.\]

Since $w$ is only supported on $T$ containing $S$, for all $T$ in the inside summation we have $|T \setminus S|=|T|-|S| \geq \delta k$, and therefore using also the $p<\psi_q$ and the definition of $w_S$,
\[\mathbb{P}_{A,Y} \left[ \text{$\exists$ $A_i \in \calA$ such that $A_i \subseteq Y, |A \cap A_i| \le qk$} \right] \leq 2^{qk} (p/\psi_q)^{\delta k} \sum_{T} w_S(T) \psi_q^{|T\setminus S|} \leq  2^{qk} (p/\psi_q)^{\delta k}.\] Since $p=o(\psi_q)$ we conclude the result.
\end{proof}

\subsection{The negative result: Proof of \Cref{thm:gen}(a)}
We now prove the negative result of the theorem. The proof follows by an interesting use of a duality trick of Talagrand between fractional covers and spread measures.

\begin{proof}
    By definition, there exists an $S \subseteq \mathcal{X}, |S| \leq qk$ such that \[\max_{w \in W_{S}} \sum_{T \subseteq \mathcal{X} } w(T)\psi_q^{|T \setminus S|}  \leq 1/2.\] Using the strong duality of linear programs, we conclude that there exists some function $\pi'$ from the set $\calA(S) \defeq \left\{ A_i \in \calA : S \subseteq A_i \right\}$ to $[0,1]$ such that
    \begin{itemize}
        \item[(a)] $\displaystyle\sum_{A_i \in \calA(S)} \pi'(A_i) = 1/2$.
        \item[(b)] For all $T$ such that $S \subseteq T$ and $|T| \geq qk$,
        \[\sum_{A_i \in \calA(S), T \subseteq A_i} \pi'(A_i) \leq (\psi_q)^{|T\setminus S|}.\]
        
        \end{itemize}Hence, setting $\pi(A_i)=2\pi'(A_i)$ for all $A_i \in \calA(S)$ we conclude the existence of a measure on $\calA(S)$ such that for all $T$ such that $S \subseteq T$ and $|T| \geq qk$ it holds
 \begin{align}\label{eq:spread}
\pi_{X \sim \pi} (T \subseteq X) \leq 2(\psi_q)^{|T\setminus S|}.
\end{align}
\begin{remark}
In the language of probabilistic combinatorics, if \eqref{eq:spread} was true for all $T$, the measure would be called $\psi_q$-spread. Yet, in our setting we only know that the ``spread" condition holds for a subset of $T$'s. Interesting, it is this spread-type of condition that will allow us to use $\pi$ as a ``non-favorable enough" prior to conclude our tight minimax rate lower bound.
\end{remark}

    We use the ``spread" $\pi$ as our prior and prove that $\liminf_{N\to\infty} \mathrm{MMSE}_N(p,\pi) \geq 1-q$ which concludes the proof. By the Nishimori identity (\Cref{lem:Nishimori}), it suffices to prove that for any $1-q>\delta>0$ it holds $\mathbb{P}(|A' \cap A| \geq (q+\delta)k)=o(1)$ where $A \sim \pi$ and $A'$ is sampled from the posterior of $A$ given $Y=A \cup X_p$.

   For any $A \in \calA(S)$, $Y \in 2^{\mathcal{X}}$ and $t \in [q+\delta,1]$, let \begin{equation}\label{e:Z.Y_2}
	Z_{Y}(A,t)
	\defeq\sum_{A_i \in \calA(S)} \pi(A_i) \mathbf{1}\{A_i \subseteq Y, |A\cap A_i| \geq  tk\}.
	\end{equation}Also denote for simplicity $Z_Y\defeq Z_Y(A,0).$ A quick application of Baye's rule gives that the posterior distribution is $\mathbb{P}(A|Y)=\pi(A)  \mathbf{1}\{A \subseteq Y\}/Z_Y, A \in \calA(S),$ and therefore we have the identity \[\mathbb{P}(|A' \cap A| \geq (q+\delta)k)=\E_{A,Y} Z_{Y}(A,q+\delta)/Z_Y.\]
Now we claim that if we establish $p^{|S|-k} \E_{A,Y} Z_{Y}(A,q+\delta)=o(1),$ then we can conclude $\E_{A,Y} Z_{Y}(A,q+\delta)/Z_Y=o(1)$. 

This implication follows by a careful ``change of measure" argument. To see this denote by $\mathbb{Q}_p$ the product law of $Y=X_p \setminus S,$ and by $\mathbb{P}_p$ the ``planted'' law of $Y=(A \setminus S) \cup (X_p \setminus S)$ for independent $A \sim \pi, X_p \sim \mathbb{Q}_p.$ Also recall that $\pi$ is always supported on sets containing $S$. Combining these, it holds for all $Y \subseteq \mathcal{X}$ that
\[ \mathbb{P}_p(Y \setminus S|A)=p^{|Y|-k} \mathbf{1}(A \subseteq Y)=p^{|S|-k} \mathbf{1}(A \subseteq Y) \mathbb{Q}_p(Y \setminus S),\] and therefore the likelihood ratio between $\mathbb{P}_p$ and $\mathbb{Q}_p$ take the following form for all $Y \subseteq \mathcal{X}$,
\[\frac{\mathbb{P}_p(Y \setminus S)}{\mathbb{Q}_p(Y \setminus S)}=Z_Y p^{|S|-k}. \] In particular, for all $\epsilon>0,$ \[\mathbb{P}_p(Z_Y \leq \epsilon p^{k-|S|})=\E_{\mathbb{Q}_p} [\frac{ \mathbb{P}_p(Y \setminus S)}{\mathbb{Q}_p(Y \setminus S)} \mathbf{1}(Z_Y p^{|S|-k} \leq \epsilon )]= \E_{\mathbb{Q}_p} [ Z_Y p^{|S|-k} \mathbf{1}(Z_Y p^{|S|-k} \leq \epsilon )] \leq \epsilon. \] Hence, for any $\epsilon>0,$ since almost surely
$Z_{Y}(A,q+\delta)/Z_Y \leq 1$, we have

\[\E_{A,Y} Z_{Y}(A,q+\delta)/Z_Y \leq \frac{1}{\epsilon} p^{|S|-k} \E_{A,Y} Z_{Y}(A,q+\delta)+ \epsilon.\]  From the last inequality by optimizing over $\epsilon>0$ we conclude \[\E_{A,Y} Z_{Y}(A,q+\delta)/Z_Y \leq 2\sqrt{p^{|S|-k} \E_{A,Y} Z_{Y}(A,q+\delta)},\] from which the claimed implication follows.

Now by linearity of expectation and \eqref{e:Z.Y_2} (for the first equality), \eqref{eq:spread} (for the second inequality) and direct counting arguments, \begin{align*} 
p^{|S|-k}  \E_{A,Y} Z_{Y}(A,q+\delta)& = \sum_{\ell \geq (q+\delta) k} \pi^{\otimes 2}(|A \cap A'|=\ell)p^{|S|-\ell}\\
& \leq \sum_{\ell \geq (q+\delta) k} \E_A \sum_{T \subseteq A, |T|=\ell} \pi(T \subseteq A')p^{|S|-\ell}\\
& \leq 2\sum_{\ell \geq (q+\delta) k} \binom{k-|S|}{\ell-|S|} (\psi_q/p)^{\ell-|S|}\\
& \leq 2\sum_{\ell \geq (q+\delta) k} (ke/(\ell-|S|)^{\ell-|S|} (\psi_q/p)^{\ell-|S|}\\
& \leq 2\sum_{\ell \geq (q+\delta) k} (e/\delta)^{\ell-|S|} (\psi_q/p)^{\ell-|S|}\\
& =O \left(((\psi_q e)/(\delta p))^{\delta k}\right).
\end{align*}Since $p=\omega(\psi_q)$ the last quantity is $o(1)$ and the proof is complete.
 \end{proof}

\section{Omitted proofs}\label{sec:ommit}
\subsection{Proof of \Cref{cor:wek-dense-improved}}
\begin{proof}[Proof of \Cref{cor:wek-dense-improved}]
    Consider a $t=0,1,\ldots,M-1$ for which \eqref{eq:jump} holds and any $p$ satisfying \eqref{eq:interval}. Now by \Cref{lem:onion-universality} for any sufficiently small $\epsilon=\epsilon_n>0$ if $q_1=|J^{(t)}|/|H|+\epsilon, q_2=|J^{(t)}|/|H|-\epsilon$ it holds \[\varphi_{q_1}=n^{-\frac{|V(J^{(t+1)}) \setminus V(J^{(t)})|}{|J^{(t+1)} \setminus J^{(t)}|}}\] and \[\varphi_{q_2}=n^{-\frac{|V(J^{(t)}) \setminus V(J^{(t-1)})|}{|J^{(t)} \setminus J^{(t-1)}|}}.\] Hence $p$ must satisfy \[(1+\delta)\varphi_{q_1} \leq p \leq (1-\delta)\varphi_{q_2}.\]By \Cref{thm:wek_dense} it therefore holds
    \[1-|J^{(t)}|/|H|-\epsilon \leq \liminf_n \MMSE_n(p) \leq \limsup_n \MMSE_n(p) \leq 1-|J^{(t)}|/|H|-\epsilon.\]As $\epsilon=\epsilon_n>0$ can be made arbitrarily small, and in particular $o(1),$ the first part follows.

    The second part follows by a similar argument noticing that for any sufficiently small $\epsilon_n=\epsilon_n>0,$ it holds
 \[n^{-\frac{|V(J^{(M)} \setminus V(J^{(M-1)})|}{|J^{(M)} \setminus J^{(M-1)}|}}=\varphi_{1-\epsilon}.\]
\end{proof}

\subsection{Proof of \Cref{cor:limit-point-inversion}}

\label{sec:limit-point-inversion}

We begin with a proof of \Cref{cor:limit-point-inversion}, which inverts \Cref{thm:main} to calculate the MMSE for a specified value of $p$.

\limitpoint*

\begin{proof}
    Consider $y \in [0,1]$ to be the limit point of the subsequence of $\alpha_{k_n}, n \in \NN.$ Observe that $G$ is a non-increasing right-continuous function. Hence for any sufficiently small constant $\eps>0$, for all large $n$, $G(\alpha_n) \leq p \leq G((1-\eps)\alpha_n)$. Now, if $y \not \in \{0,1\},$ then we can conclude by the fact that $G$ is decreasing that for all $\eps>0$ it holds \[\limsup_{n} p_{k_n}/(G(y-\eps))_{k_n} \leq 1 \leq \liminf_n p_{k_n}/(G(y+\eps))_{k_n}.\] Using \Cref{thm:wek_dense} (which trivially generalizes over subsequences) this allows us to conclude \begin{align*}
       1-y-\eps \leq \liminf_n \MMSE_{k_n}(p) \leq  \limsup_n \MMSE_{k_n}(p) \leq 1-y+\eps.
    \end{align*}
    Since $\eps>0$ is arbitrary, \eqref{eq:mmse_fin} follows for $y \in (0,1)$. If $y=0$, then we know for each $\eps>0,$ it must hold $\liminf_n p_{k_n}/(G(\epsilon))_{k_n} \geq 1.$ Therefore from \Cref{thm:wek_dense}, \begin{align*}
       1-\eps \leq \liminf_n \MMSE_{k_n}(p).
    \end{align*}Since $\eps>0$ is arbitrary, \eqref{eq:mmse_fin} follows. The case $y=1$ follows similarly to the case $y=0$.
    
\end{proof}

\subsection{Proofs of lemmas for the $\rho$ function}

In this section, we prove some properties of the $\rho$ function, restated for convenience.

\densitymodularity*

\begin{proof}
    We have 
    \begin{align*}
        \rho(A\cup B|A) &= \frac{|(A\cup B) \setminus A|}{v(A\cup B) - v(A)} \\
        &= \frac{|B\setminus (A\cap B)|}{v(B) - |V(A)\cap V(B)|},    \end{align*}
    so the lemma is a consequence of the simple inequality \[|V(A)\cap V(B)| \geq |V(A \cap B)|=v(A\cap B). \qedhere \]
\end{proof}

\lembalance*

\begin{proof}
    For any $T'$  such that $T \subseteq T' \subseteq K$ we have
    \begin{align*}
        \rho(K|T') = \frac{|K \setminus T' |}{v(K)-v(T')} = \frac{(|K|-|T|) - (|T'|-|T|)}{(v(K)-V(T)) - (v(T')-v(T))} \ge \frac{|K \setminus T |}{v(K)-v(T)} =\rho(K|T),
    \end{align*}
    where the final inequality is equivalently with $\rho(T'|T) \le \rho(K|T)$.
\end{proof}

\maximizerunion*

\begin{proof}
Let $Q'=J'\setminus A,Q''=J''\setminus A.$
    Then we have 
    \begin{align}\label{eq:union}
        \rho(J'\cup J''|A) &= \frac{|Q'| + |Q''| - |Q\cap Q''|}{|V(Q')| + |V(Q'')| - |V(Q')\cap V(Q'')|} \\
        &\geq \frac{|Q'| + |Q''| - |Q\cap Q''|}{v(Q') + v(Q'') - v(Q'\cap Q'')}.
    \end{align}
By the balancedness of one of $J',J''$,
\[ 
\rho(J' \cap J''|A)=\frac{|Q' \cap Q''|}{v(Q' \cap Q'')} \leq \max\left\{\frac{|Q'| }{v(Q')}, \frac{|Q''| }{v(Q'')} \right\}.
\]
Combining the two last displayed equations with \eqref{eq:union} we conclude 
\[
\rho(J'\cup J''|A) \geq \min\left\{\frac{|Q'|}{v(Q') }, \frac{|Q''|}{v(Q'') }\right\}=\min\{\rho(J'|A), \rho(J''|A)\},
\]
as desired.
\end{proof}

\union*

\begin{proof}
For all $C \supseteq B,$ we have \[\rho(B|A)=\frac{|B|-|A|}{|V(B)|-|V(A)|} \geq \rho(C|A) =\frac{|C|-|A|}{|V(C)|-|V(A)|} \]and therefore \[\frac{|B|-|A|}{|V(B)|-|V(A)|} \geq \frac{|C|-|B|+(|B|-|A|)}{(|V(C)|-|V(B)|)+(|V(B)|-|V(A)|)},\]which rearranging gives \[\rho(C|B)=\frac{|C|-|B|}{|V(C)|-|V(B)|} \leq \rho(B|A) =\frac{|B|-|A|}{|V(B)|-|V(A)|} \] 

Now, consider $J^{(t)}, t=0,1,\ldots,M$ the onion decomposition of a graph $H$. Then for any $t=0,1,\ldots,M-1,$ applying the established inequality for $A\defeq J^{(t-1)}$, which gives by definition $B\defeq J^{(t)},$ and $C\defeq J^{(t+1)} \supseteq J^{(t)},$ implies \[\rho(J^{(t+1)}|J^{(t)}) \leq \rho(J^{(t)}|J^{(t-1)}). \qedhere \]
\end{proof}

\parhead{Acknowledgements.} We would like to thank Ewan Davies and Will Perkins for organizing the 2024 Georgia Tech Summer School on Probability, Algorithms, and Inference, during which this project started. 

%% file: real_appendix_iz.tex
\appendix
\section{Auxiliary Lemmas}
In this section we prove some standard results that are needed for the main body of this work.

We start with the Nishimori identity we use often in this work.
\begin{lemma}[Nishimori's identity]\label{lem:Nishimori}
    Suppose a Bayesian inference setting where the parameter $\theta$ is sampled from some prior $\pi$ supported on $S^{N-1}$ , i.e., $\theta \sim \pi,$ and we have observations $Y \sim P_{\theta}$. Let also $\pi'(Y)$ be the posterior distribution of $\theta$ given $Y$. Then, 
    \[\MMSE(\theta):=\E[\|\theta-\E[\theta|Y]\|^2_2]=1-\E_{\theta \sim \pi, Y \sim P_{\theta}, \theta' \sim \pi'(Y)}[\langle \theta,\theta'\rangle].\] In particular, for the planted subgraph model for a graph $H$ it holds
    \[\MMSE_n(p)=1-\frac{1}{|H|}\E[|H^* \cap H'|],\]where $H'$ is a uniform random copy of $H$ in the observed graph $G.$
\end{lemma}

\begin{proof}
    See \cite[Lemma 2]{niles2023all} for the proof of the identity. The application for the normalized MMSE of the planted subgraph model follows immediately from the fact that the posterior in the planted subgraph model is the uniform measure on the copies of $H$ in the observed graph $G$.
\end{proof}

The following is a standard lemma on the MMSE curve.
\begin{lemma}\label{lem:MMSE_poly}
    For any graph $H=H_n$ and for any $n, p=p_n \in (0,1)$, under the planted subgraph model corresponding to $H$ the function $\MMSE_n(p)$ is a non-decreasing polynomial function of $p$.
\end{lemma}

\begin{proof}
   By the Nishimori identity, \[\MMSE_n(p)=1-\frac{1}{|H|}\E[|H' \cap H^*|],\]where $H^*$ is the planted copy of $H$ in $G$ and $H'$ is a uniform random copy of $H$ in the observed $G$. Expanding a direct argument gives,
   \begin{align}\label{eq:MMSE_formula}
   \MMSE_n(p)=1-\frac{M_H}{|H|}\sum_{\ell} \ell  \P(|H" \cap H^*|=\ell) p^{|H|-\ell},
   \end{align}where now $M_H$ is the total number of copies of $H$ in $K_n,$ and $H^*, H''$ are now independent samples from the uniform distribution over \emph{all} copies of $H$ in $K_n.$ Notice that clearly \eqref{eq:MMSE_formula} implies that $\MMSE_n(p)$ is a polynomial function in $p$.

   The fact that $\MMSE_n(p)$ is non-decreasing in $p$ follows by a simple application of the data processing inequality, see e.g. \cite[Lemma 3.2]{MNSSZ23}.
\end{proof}

We now present a simple identity on the first moment threshold of $H$.
\begin{lemma}\label{lem:1mm_wek}
    Suppose $H$ is a weakly dense graph. Then it's first moment threshold $p_{\mathrm{1M}}(H)$ satisfies 
    \[p_{\mathrm{1M}}(H)=(1+o(1))n^{-v(H)/|H|}.\]
\end{lemma}
We refer the reader to \cite[Lemma 4.2]{MNSSZ23} for the proof.

Finally, we show that, in the context of monotone properties given by graph inclusions, the minimax prior is given by the uniform distribution.
\begin{lemma}\label{lem:unif-is-worst-case}
    Let $\calA \subseteq \{0,1\}^{\binom{n}{2}}$ be a subgraph inclusion property, that is, there exists a graph $H \subseteq \binom{[n]}{2}$ such that $G \in \calA$ if and only if $G$ contains an isomorphic copy of $H$.\footnote{Isomorphic here is in the sense of \emph{graph} isomorphism, i.e., relabeling of the vertices.} Then, the maximizer $\pi^*$ of $\pi\mapsto \MMSE_N(p, \pi)$ in \eqref{eq:minimax} is given by the uniform distribution over all copies of $H$ in $K_n = \binom{[n]}{2}.$
\end{lemma}
\begin{proof}
    Let $\sigma\in S_n$ be any permutation of the vertices $[n]$ and, for a prior $\pi$, define the action of $\sigma$ on $\pi$ as follows:
    \[
    \pi_\sigma(H') = \pi(H'_{\sigma}),
    \]
    where $H'_\sigma$ is the copy of $H$ obtained by relabeling the vertices of $H'$ according to $\sigma.$ It's clear by symmetry that, for all $\pi$ and $\sigma,$ we have 
    \begin{align*}
        \MMSE_N(p,\pi) = \MMSE_N(p,\pi_\sigma).
    \end{align*}
    Now suppose we choose $\sigma$ according to the uniform distribution in $S_n.$ Since we have a Markov chain
    \[
    \sigma \to \pi_\sigma \to X_p \cup H',\; H'\sim \pi_\sigma,
    \]
    by the data processing inequality, for any $\pi,$ we have
    \[
    \MMSE_N(p, \pi) =\E_\sigma \MMSE_N(p, \pi_\sigma) \leq \MMSE_N(p, \E_\sigma\pi_\sigma).
    \]
    Finally, since $S_n$ acts transitively on all copies of $H$ in $K_n$, we have that $\E_\sigma\pi_\sigma$ is the uniform distribution for all $\pi$, and hence $\pi^*$ is uniform.
\end{proof}